\documentclass{article}

\usepackage[T1]{fontenc}
\usepackage[english]{babel}
\usepackage[babel]{microtype}
\usepackage[a4paper, margin=1.1in]{geometry}

\usepackage{amsmath,amssymb,amsthm,amsrefs}

\usepackage{mathrsfs}

\usepackage{mathtools}
\usepackage{commath}

\usepackage{bbm}

\usepackage{cases}
\usepackage{enumitem}

\usepackage{centernot}
\usepackage{booktabs}
\usepackage{multirow}
\usepackage{comment}
\usepackage{float}
\usepackage{scalerel}
\usepackage{algorithm,algorithmic,float}

\usepackage[pdftex, pdfborderstyle={/S/U/W 0}]{hyperref}
\hypersetup{
	colorlinks=true,
	linkcolor=magenta,
	citecolor=cyan,
}

\usepackage{cleveref}

\usepackage{etoolbox}

\usepackage{subfiles}
\usepackage{pdfpages}

\usepackage{tikz}
\usepackage{caption}
\usepackage{subcaption}

\usepackage{lineno}

\numberwithin{equation}{section}

\theoremstyle{plain}

\newtheorem{theorem}    {Theorem}[section]
\newtheorem{lemma}      [theorem] {Lemma}

\newtheorem{proposition}[theorem] {Proposition}
\newtheorem{corollary}  [theorem] {Corollary}

\theoremstyle{definition}
\newtheorem{definition} [theorem] {Definition}

\newtheorem{claim}      [theorem] {Claim}

\newtheorem{remark}    [theorem]{Remark}

\makeatletter
\newtheorem*{rep@theorem}{\rep@title}
\newcommand{\newreptheorem}[2]{%
	\newenvironment{rep#1}[1]{%
		\def\rep@title{#2 \ref{##1}}%
		\begin{rep@theorem}}%
		{\end{rep@theorem}}}
\makeatother

\newreptheorem{theorem}{Theorem}
\newreptheorem{lemma}{Lemma}
\newreptheorem{claim}{Claim}

\renewcommand{\le}{\leqslant}
\renewcommand{\leq}{\leqslant}
\renewcommand{\ge}{\geqslant}
\renewcommand{\geq}{\geqslant}

\let\oldexists\exists
\let\exists\relax
\DeclareMathOperator{\exists}{\oldexists}
\let\oldforall\forall
\let\forall\relax
\DeclareMathOperator{\forall}{\oldforall}

\undef{\set}
\DeclarePairedDelimiter{\set}{\lbrace}{\rbrace}

\undef{\mod}
\newcommand{\mod}[1]{\ (\mathrm{mod}\ #1)}

\undef{\abs}
\DeclarePairedDelimiterX{\abs}[1]
{\lvert}{\rvert}{\ifblank{#1}{\,\cdot\,}{#1}}
\undef{\norm}
\DeclarePairedDelimiterX{\norm}[1]
{\lVert}{\rVert}{\ifblank{#1}{\,\cdot\,}{#1}}
\DeclarePairedDelimiterX{\inner}[2]
{\langle}{\rangle}{\ifblank{#1}{\,\cdot\,}{#1},\ifblank{#2}{\,\cdot\,}{#2}}

\DeclareMathOperator{\Bin}{Bin}

\DeclareMathDelimiter{\given}
{\mathbin}{symbols}{"6A}{largesymbols}{"0C}
\DeclareMathOperator{\Prob}{\mathbb{P}}
\DeclarePairedDelimiterXPP{\prob}[1]
{\Prob}{\lparen}{\rparen}{}
{\renewcommand{\given}{\nonscript\;\delimsize\vert\nonscript\;\mathopen{}}
	\ifblank{#1}{\,\cdot\,}{#1}}
\DeclareMathOperator{\Expec}{\mathbb{E}}
\DeclarePairedDelimiterXPP{\expec}[1]
{\Expec}{\lparen}{\rparen}{}
{\renewcommand{\given}{\nonscript\;\delimsize\vert\nonscript\;\mathopen{}}
	\ifblank{#1}{\,\cdot\,}{#1}}
\DeclareMathOperator{\Var}{Var}
\DeclarePairedDelimiterXPP{\var}[1]
{\Var}{\lparen}{\rparen}{}
{\renewcommand{\given}{\nonscript\;\delimsize\vert\nonscript\;\mathopen{}}
	\ifblank{#1}{\,\cdot\,}{#1}}
\DeclareMathOperator{\Cov}{Cov}
\DeclarePairedDelimiterXPP{\cov}[2]
{\Cov}{\lparen}{\rparen}{}{#1,#2}

\newcommand{\EE}{\mathbb{E}}

\newcommand{\PP}{\mathbb{P}}

\newcommand{\cA}{\mathcal{A}}

\newcommand{\cE}{\mathcal{E}}
\newcommand{\cF}{\mathcal{F}}

\newcommand{\cM}{\mathcal{M}}

\newcommand{\cP}{\mathcal{P}}

\DeclareMathOperator{\END}{END}
\DeclareMathOperator{\e}{end}
\DeclareMathOperator{\m}{mid}

\DeclareMathOperator{\Des}{Des}

\newcommand{\IN}[2]{\text{IN}\left(#1; #2 \right)}
\newcommand{\OUT}[2]{\text{OUT}\left(#1; #2 \right)}

\floatname{algorithm}{Procedure}

\usepackage{footnote}
\makesavenoteenv{algorithm}

\title{Minimum degree edge-disjoint Hamilton cycles in random directed graphs}
\author{Asaf Ferber\thanks{Department of Mathematics, University of California, Irvine. Email: \href{mailto:asaff@uci.edu} {\nolinkurl{asaff@uci.edu}}. Research supported in part by NSF Awards DMS-1954395 and DMS-1953799.} \hspace{10mm} Adva Mond\thanks{Department of Mathematics, King's College London, London, United Kingdom. Email: \href{mailto:adva.mond@kcl.ac.uk}{\nolinkurl{adva.mond@kcl.ac.uk}}}}
\date{}

\begin{document}
	
	\maketitle
	
	\begin{abstract}
		In this paper we consider the problem of finding ``as many edge-disjoint Hamilton cycles as possible'' in the binomial random digraph $D_{n,p}$.
		We show that a typical $D_{n,p}$ contains precisely the minimum between the minimum out- and in-degrees many edge-disjoint Hamilton cycles, given that $p\geq \log^{15} n/n$, which is optimal up to a factor of poly$\log n$. 
		
		Our proof provides a randomized algorithm to generate the cycles and uses a novel idea of generating $D_{n,p}$ in a sophisticated way that enables us to control some key properties, and on an ``online sprinkling'' idea as was introduced by Ferber and Vu.
	\end{abstract}
	
	
	\section{Introduction}
	\emph{At most how many edge-disjoint Hamilton cycles does a given graph/directed graph contain?}\\ 
	Recall that a \emph{Hamilton cycle} in a directed graph (digraph) is a directed cycle passing through all of its vertices.
	Even though a Hamilton cycle is a relatively simple structure, determining whether a certain graph is \emph{Hamiltonian} (that is, contains a Hamilton cycle) was included in the list of $21$ $\mathcal{NP}$-hard problems by Karp~\cite{karp1972reducibility}. Therefore, it is interesting to find general sufficient conditions for the existence of Hamilton cycles in (di)graphs under some mild assumptions.
	
	Let $\psi(G)$ denote the maximum number of edge-disjoint Hamilton cycles in a given (di)graph $G$. The problem of estimating $\psi(G)$ for (di)graphs with large minimum degree has a long and rich history and we briefly sketch some notable results under four settings: $(1)$ dense graphs, $(2)$ dense digraphs, $(3)$ random graphs, and $(4)$ random digraphs.

	\paragraph{Dense graphs.} 
	Since a Hamilton cycle is $2$-\emph{regular} (that is, every vertex has degree exactly $2$), we obtain the trivial upper bound $\psi(G)\leq \lfloor \delta(G)/2\rfloor$, where $\delta(G)$ is the minimum degree of $G$. A classical theorem by Dirac~\cite{Dirac} asserts that a graph $G$ on $n \ge 3$ vertices\footnote{Unless stated otherwise, $n$ always stands for the number of vertices in a (di)graph.} and of minimum degree at least $n/2$ is Hamiltonian, and in particular, $\psi(G)\geq 1$. 
	
	In the 1970s, Nash-Williams~\cite{nashwilliams} showed that if a graph $G$ satisfies the Dirac condition (that is, $\delta(G)\geq n/2$), then it already contains at least $\left\lfloor \frac{5}{224}n \right\rfloor$ edge-disjoint Hamilton cycles.
	Ferber, Krivelevich, and Sudakov~\cite{ferber2017counting} proved that if $\delta(G)\geq (1/2+o(1))n$ then $G$ contains at least $(1-o(1))\text{reg}_{\text{even}}(G)/2$ edge-disjoint Hamilton cycles, where $\text{reg}_{\text{even}}(G)$ is the largest \emph{even} number $d$ such that $G$ contains a spanning $d$-regular subgraph.
	Csaba, K\"{u}hn, Lo, Osthus, and Treglown~\cite{1factorization} proved that a graph $G$ of minimum degree $\delta \ge n/2$, where $n$ is sufficiently large, contains $\text{reg}_{\text{even}}(n, \delta)/2$ edge-disjoint Hamilton cycles, where $\text{reg}_{\text{even}}(n,\delta)$ is the minimum $\text{reg}_{\text{even}}(G')$ over all $n$-vertex $G'$ with $\delta(G') = \delta$. To show that an $n$-vertex graph $G$ with $\delta(G)\geq n/2$ contains $\text{reg}_{\text{even}}(G)/2$ edge-disjoint Hamilton cycles is still open (and was conjectured in \cite{kuhn2013hamilton}).
	It is also worth mentioning that in ~\cite{1factorization} Csaba, K\"{u}hn, Lo, Osthus, and Treglown proved that every $d$-regular graph, where $d\ge n/2$ and $n$ is sufficiently large, contains $d/2$ many edge-disjoint Hamilton cycles.

	\paragraph{Dense digraphs.}
	
	Since in a directed Hamilton cycle every vertex has in/out-degrees exactly $1$, we obtain the trivial upper bound $\psi(D)\leq \delta^{\pm}(D)$, where $\delta^{\pm}(D)$ is the minimum between the minimum in-degree and the minimum out-degree of $D$, denoted by $\delta^-(D)$ and by $\delta^+(D)$, respectively. A digraph analogue for Dirac's theorem was proved by Ghouila-Houri~\cite{ghouilahouri}, who showed that a digraph $D$ with $\delta^{\pm}(D)\geq n/2$ is Hamiltonian, and hence, $\psi(D)\geq 1$.
	
	In the late $60$'s, Kelly conjectured (see the excellent surveys \cites{kuhn2012survey,kuhn2014hamilton} and the references therein) that every regular \emph{tournament} (where a tournament is an orientation of the complete graph) contains $(n-1)/2$ edge-disjoint Hamilton cycles, which is clearly best possible.
	Kelly’s Conjecture attracted a lot of attention in the past few decades, until it was settled for large tournaments in a remarkable (and over 90 pages long) paper by K\"{u}hn and Osthus \cite{kuhn2013hamilton}.
	
	A tournament is a special case of a more general set of directed graphs, so called \emph{oriented graphs}, which are obtained by orienting the edges of a simple graph.
	After a long line of research, Keevash, K\"uhn and Osthus~\cite{keevash2009exact} proved that $\delta^{\pm}(G) \ge \left\lceil \frac{3n-4}{8} \right\rceil$ implying the existence of a Hamilton cycle. A construction showing that this is tight was obtained much earlier by H\"aggkvist~\cite{haggkvist1993hamilton}.
	In~\cite{kuhn2014applications}, among other interesting corollaries to the methods developed in~\cite{kuhn2013hamilton}, K\"uhn and Osthus showed that a $d$-regular digraph $D$, where $d \ge (1/2+o(1))n$ and $n$ is sufficiently large, satisfies $\psi(D) = d$.
	
	\paragraph*{Random graphs.}
	Let $G_{n,p}$ be the binomial random graph on vertex-set $[n]:=\{1,\ldots,n\}$, obtained by adding each unordered pair $xy\in \binom{[n]}{2}$ as an edge with probability $p$, independently at random. A cornerstone result in random graphs was obtained by P\'{o}sa~\cite{posa} in 1976, who introduced the so-called rotation-extension technique,  and proved that a typical $G_{n,p}$ is Hamiltonian for $p = \Theta\left(\frac{\log n}{n} \right)$. Based on P\'osa's technique, Bollob\'{a}s~\cite{bollobas1984evolution}, and independently, Koml\'{o}s and Szemer\'{e}di~\cite{komlos1983limit}, determined the sharp threshold (for Hamiltonicity) to be $p = \frac{\log n + \log \log n + \omega(1)}{n}$.
	\footnote{This is indeed optimal, as for $p = \frac{\log n +\log \log n - \omega(1)}{n}$ we have that with high probability $\delta\left(G_{n,p} \right) \le 1$, and therefore the graph is not Hamiltonian.}
	
	Later on, Bollob\'{a}s and Frieze~\cite{bollobas1983matchings} extended the above and proved that for $p=\frac{\log n+O(\log\log n)-\omega(1)}{n}$ (for this range of $p$, a typical $G_{n,p}$ has minimum degree $O(1)$) we typically have $\psi\left(G_{n,p} \right) = \lfloor \delta\left(G_{n,p} \right)/2\rfloor$.
	For larger values of $p$, combining all the results of Frieze and Krivelevich~\cites{frieze2005packing,frieze2008two}, Knox, K\"{u}hn, and Osthus~\cites{knox2012approximate,knox2015edge}, Ben-Shimon, Krivelevich, and Sudakov~\cite{ben2011resilience}, Krivelevich and Samotij~\cite{krivelevich2012optimal}, and K\"{u}hn and Osthus~\cite{kuhn2014applications} we conclude that 
	typically $\psi\left(G_{n,p} \right) = \left\lfloor \delta\left(G_{n,p} \right)/2 \right\rfloor$ for all values of $p$. 
	
	It is worth mentioning that all the above mentioned work was heavily based on the rotation-extension technique, which unfortunately does not apply (at least as is) when working with digraphs.
	
	\paragraph*{Random digraphs}
	Let $D_{n,p}$ denote the binomial random digraph on vertex-set $[n]$, obtained by adding each ordered pair $xy\in [n]^2$ as an edge with probability $p$, independently at random (self loops are not allowed). The question of Hamiltonicity in the random digraph $D_{n,p}$ was firstly addressed by McDiarmid~\cite{mcdiarmid1980clutter}, who proved, based on an elegant coupling argument, that 
	$$\Pr\left[G_{n,p} \text{ is Hamiltonian} \right] \le \Pr\left[D_{n,p} \text{ is Hamiltonian} \right].$$
	Combining this with  \cites{bollobas1984evolution,komlos1983limit} we obtain that $D_{n,p}$ is typically Hamiltonian for $p \ge \frac{\log n + \log \log n + \omega(1)}{n}$. Frieze~\cite{frieze1988algorithm} later determined the correct Hamiltonicity threshold for $D_{n,p}$ to be $p = \frac{\log n + \omega(1)}{n}$. \footnote{Also here, this is optimal, as for $p = \frac{\log n - \omega(1)}{n}$ we typically have $\delta^{\pm}\left(D_{n,p} \right) = 0$, and therefore the graph is not Hamiltonian.} Moreover, in the same paper Frieze proved that for $p=\frac{\log n +O(\log\log  n)-\omega(1)}{n}$ we typically have $\psi(D_{n,p})=\delta^{\pm}(D_{n,p})$.
	Ferber, Kronenberg, and Long~\cite{ferber2017packing} proved an approximate result by showing that for $p=\omega\left(\frac{\log^4n}{n}\right)$ we typically have $\psi\left(D_{n,p} \right) = (1-o(1))\delta^{\pm}\left(D_{n,p} \right)$. For $p=\Theta(1)$ it follows from the methods in~\cites{kuhn2013hamilton}  that typically $\psi\left(D_{n,p} \right) = \delta^{\pm}\left(D_{n,p} \right)$ which is best possible. 
	
	In this paper we address this problem for all values $p\geq \frac{\log^{15}n}{n}$. 
	\begin{theorem}
		\label{thm:main3}
		For $p\geq \frac{\log^{15} n}{n}$ with high probability we have 
		\[\psi\left(D_{n,p} \right) = \delta^{\pm}\left(D_{n,p} \right).\]
	\end{theorem}
	As we have already mentioned, the case where $p = \Theta(1)$ follows from methods in~\cite{kuhn2013hamilton}, so it is sufficient to prove \Cref{thm:main3} for $\frac{\log^{15} n}{n} \le p \le \varepsilon$, for some specific constant $\varepsilon > 0$.

	Our proof provides a randomized algorithm that with high probably outputs $\delta^{\pm}(D_{n,p})$ edge-disjoint Hamilton cycles. Our argument is based on an approach introduced in~\cite{ferber2018counting} to sophisticatedly generate $D_{n,p}$ in a way that enables us to make use of partial randomness along the process, and on an ``online sprinkling'' idea that was introduced in \cite{OnlineSprinkling}. The exact details are given below.

	\subsection{Notation}
	Given a digraph $D$ and two subsets of vertices $A,B \subset V(D)$ we denote by $E(A,B)$ the set of directed edges from vertices in $A$ to vertices $B$, and we let $e(A,B) = \big|E(A,B) \big|$.
	If either $A$ or $B$ consists of a single vertex $v$, we simply write $E(v,B)$ or $E(A,v)$.
	For a directed path or cycle $P$ and a vertex $u \in P$, we denote by $u^P_+$ and by $u^P_-$ its successor and predecessor in $P$, respectively.
	For a subset of vertices $U \subset V(P)$ we denote $U^P_- \coloneqq \left\{u^P_- : u \in U \right\}$, $U^P_+ \coloneqq \left\{u^P_+ : u \in U \right\}$.
	When the path/cycle is clear from the context, we sometimes omit the superscript and simply write $u_-, u_+$ and $U_-, U_+$.
	
	
	\section{Proof outline}
	\label{sec:ProofOutline}
	
	Our proof, roughly, consists of the following two main phases: 
	
	\paragraph*{Phase 1} Finding $\delta^{\pm}(D_{n,p})$ edge-disjoint $1$-\emph{factors}, where a $1$-factor is a collection of vertex disjoint cycles covering all the vertices of the graph.
	\paragraph*{Phase 2} Converting each $1$-factor into a Hamilton cycle using disjoint sets of edges (in particular, the resulting Hamilton cycles are edge-disjoint). 
	\bigskip
	
	In order to execute our plan, we would start Phase 1 by generating a subdigraph $D'\subseteq D_{n,p}$ such that $(1)$ $\delta^{\pm}(D')=\delta^{\pm}(D_{n,p})$, and $(2)$ $D'$ contains $\delta^{\pm}(D_{n,p})$ edge-disjoint $1$-factors.
	Moreover, we would generate $D'$ in a way that $D_{n,p}=D'\cup D''$ where $D''$ is a carefully defined random digraph in which we expose edges with probability $p_1$, for some small $p_1$ to be determined later.
	This extra randomness of $D''$ would be used in Phase 2 to convert the $1$-factors into Hamilton cycles using an ``online sprinkling'' method.
	Since $D'$ contains $\delta^{\pm}(D_{n,p})$ edge-disjoint $1$-factors, and since $D_{n,p}$ is ``roughly regular'', it follows that $|E(D'')|$ is typically ``small''.
	In particular it means that $p_1$ cannot be too large (and on the other hand it cannot be too small, as we would use only edges of $D''$ to convert $1$-factors into Hamilton cycles in Phase 2).
	Since $p_1$ is ``relatively small'', it could be useful if each $1$-factor would not have ``too many'' cycles (and hence, we would not need to use ``too many'' new edges in order to convert it into a Hamilton cycle), and indeed, we show that this is the case.
	The way that we generate $D'$ (and hence, $D_{n,p}$) is based on a sophisticated multiple exposure argument that enables us to find the desired family of edge-disjoint $1$-factors, each of which consists of at most $O(\log n)$ cycles (see~\Cref{lem:fewcycles}).
	
	\subsection{Phase 1}
	\paragraph*{Generating $D'$ along with the desired $1$-factors.}
	We use a small tweak on a natural (and known) bijection between digraphs and bipartite graphs as follows:
	Let $B = (X\cup Y, E)$ be a bipartite graph where $X$ and $Y$ are two disjoint copies of $[n]$. For convenience we label $X=\{x_1,\ldots,x_n\}$ and $Y=\{y_1,\ldots,y_n\}$, and let $\pi\in S_n$ be any permutation. Let $D_{\pi}:=D_{\pi}(B)$ be a digraph with vertex-set $V(D_{\pi})=[n]$, and with edge-set, $E(D_{\pi})$, consisting of all directed edges $\overrightarrow{i\pi(j)}$ where $x_iy_j\in E(B)$, except for loops (that is, we erase directed edges of the form $\overrightarrow{ii}$).  
	
	In Procedure~\ref{procedure:D} we describe how to generate $B$ and $D'$. Let $p_0, p_1 \in [0,1]$ be such that $(1-p_0)(1-p_1) = 1-p$, and $p_1 = \sqrt{\frac{p}{n \log^4 n}}$.
	\begin{algorithm}
		\caption{Generate $D'$ with $\delta^{\pm}(D')$ edge-disjoint $1$-factors}
		\label{procedure:D}
		\begin{algorithmic}[1]
			\STATE \textbf{First exposure.} Generate a bipartite $B'=(X\cup Y,E')$ by adding each edge $xy\in X\times Y$ with probability $p_0$, independently at random.
			\STATE Let $x^+ \in X$ and $y^- \in Y$ be any two vertices of minimum degree in $X$ and $Y$, respectively.
			\STATE \textbf{Second exposure.} Each edge from $\{x^+y, xy^- \mid x\in X,y\in Y\}$, is being added to $E'$ with probability $p_1$, independently at random. Let $B= (X\cup Y, E)$ denote the obtained bipartite graph.\STATE \textbf{Find edge-disjoint perfect matchings.} Let $\delta \coloneqq \min\{\deg_{B}(x^+), \deg_{B}(y^-) \}$.
			If there are $\delta$ edge-disjoint perfect matchings in $B$ declare SUCCESS and continue.
			Otherwise end procedure and output FAILURE.
			\STATE \textbf{Convert to a digraph.} Let $\pi \in S_n$ be a randomly chosen permutation, and define $D' \coloneqq D_{\pi}(B)$.
		\end{algorithmic}
	\end{algorithm}
	
	A moment's thought reveals that $D'$ can be coupled as a subdigraph of $D_{n,p}$.
	In particular, one can easily see that if $D''$ is a random graph on the vertex-set $[n]$ obtained by exposing all the edges in $\left\{\overrightarrow{uv} ~:~ u \in [n]\setminus \{x^+ \}, \; v \in [n]\setminus \{\pi(y^-) \} \right\}$ independently with probability $p_1$, then $D_{n,p}=D'\cup D''$.
	
	The (typical) success of Step $4$ in Procedure~\ref{procedure:D} follows from \Cref{lem:disjointPM}.
	Now, fix any collection of edge-disjoint perfect matchings, and observe that it corresponds to a collection of edge-disjoint $1$-factors in $D'$ (no matter which $\pi$ we choose).
	Moreover, by choosing $\pi\in S_n$ uniformly at random, each of the perfect matchings in $B$ corresponds to a \emph{random} $1$-factor.
	In particular, by the well known fact that the expected number of cycles in a random $1$-factor is $O(\log n)$, and by concentration of measure (see \Cref{lem:fewcycles}) we conclude that, typically, all the matchings in our collection correspond to $1$-factors with at most (say) $\log^2n$ many cycles.
	For this property of the set of $1$-factors we get in $D'$, choosing a random uniform permutation is a crucial step in our proof.
	
	By the choice of the parameter $p_0$ we already get that typically $\delta = \delta^{\pm}(D_{n,p})$ (see \Cref{lem:mindegDnp}).
	We then show that typically step $4$ in Procedure~\ref{procedure:D} does not output FAILURE (see \Cref{lem:disjointPM}).
	This altogether imply that with high probability, once $D'$ is generated, we already get it with $\delta$ edge-disjoint $1$-factors in it.
	
	\subsection{Phase 2}
	\paragraph*{Generating $D''$ and converting $1$-factors into Hamilton cycles.}
	We expose the remaining edges using an ``online sprinkling'' technique (this technique was introduced in~\cite{OnlineSprinkling}) as follows: for a carefully chosen $q\in [0,1]$, we wish to convert each $1$-factor into a Hamilton cycle, one at a time,  in an ``economical'' way by exposing edges only when needed with probability $q$, independently at random. 
	Then, at the end of the process, we will show that typically, for every pair $(x,y)\in [n]^2$, we have not tried to expose $(x,y)$ more than $p_1/q$ times (see~\Cref{lem:uvProb}) and therefore, the ``newly exposed'' edges can be coupled as a subdigraph of $D''$ (for more details, see \cite{OnlineSprinkling}). 
	
	In order to ``tailor'' cycles in a given $1$-factor to one another we choose a special vertex from each cycle (in each $1$-factor), from which we expose edges to another cycle to initiate our directed version of P\'{o}sa's rotation-extention technique, a process which is similar in spirit to the ``double-rotation'' process introduced in \cite{frieze1988algorithm}.
	However, since some cycles are possibly very short, a vertex might be chosen ``too many times'', which could be a problem as we might expose some edges incident to this vertex more than $p_1/q$ many times, and we will not be able to couple the resulting digraph with a subdigraph of $D''$.
	To overcome this issue we prove \Cref{lem:keygnrl}, which is a key lemma in our proof. 
	Roughly speaking, this lemma states that typically no vertex can be chosen ``too many times'' (intuitively, if the $1$-factors were truly random and independent, then this is clearly true).
	Interestingly, the argument in the proof of this lemma is one of these rare examples where a second moment calculation is too weak but a third moment calculation suffices.
	
	
	\section{\texorpdfstring{Phase 1: Generating $D'$ along with $\delta^{\pm}(D')$ edge-disjoint $1$-factors}{Phase 1: Generating D' along with delta(D') edge-disjoint 1-factors}}
	
	Throughout the paper we make an extensive use of the known concentration inequality by Chernoff, bounding the lower and the upper tails of the Binomial distribution (see, e.g.,~\cites{AlonSpencer,JRLrandomgraphs}).
	\begin{theorem}[Chernoff bound]
		\label{Chernoff}
		Let $X\sim \Bin(n, p)$ and let $\EE[X] = \mu$.
		Then
		\begin{itemize}
			\item $\Pr[X < (1-\eta)\mu]<e^{-\eta^2\mu/2}$ for every $0 < \eta < 1$;
			\item $\Pr[X>(1+\eta)\mu] < e^{-\eta^2\mu/(2+\eta)}$ for every $\eta > 0$.
		\end{itemize}
	\end{theorem}

	The procedure for generating $D'$ along with the $1$-factors was already described in Phase 1 of the proof outline (see Procedure~\ref{procedure:D}), and in this section we prove that this is indeed possible.
	
	We use the following result by Gale and Ryser (see, e.g.,~\cite{lovasz2007combinatorial}) to find edge-disjoint perfect matchings in an auxiliary bipartite graph.
	\begin{theorem}[Gale-Ryser]
		\label{lem:GaleRyser}
		A bipartite graph $G = (X \cup Y, E)$ with $|X| = |Y| = n$ contains an $r$-factor if and only if for all $A \subseteq X$ and $B \subseteq Y$, the following holds,
		\begin{align}
		\label{eq:GaleRyser}
		e_G(A, B) \ge r\left(|A| + |B| - n \right).
		\end{align}
	\end{theorem}
	
	For a graph $G$ we denote by $\delta_2(G)$ its second smallest degree.
	Another crucial result that we need is the following corollary of a theorem by Bollob\'{a}s (Theorem 3.15 in~\cite{bollobas1998random}), simplified for our needs.
	By \emph{with high probability} we mean with probability tending to $1$ as $n$ tends to infinity.
	\begin{theorem}
		\label{thm:DegGap}
		Let $\omega\left(\frac{\log^2 n}{n}\right)= p_0 \leq 1/2$.
		Then with high probability we have
		\begin{align*}
		\delta_2(B_{n,p_0}) - \delta(B_{n,p_0}) \ge \frac{\sqrt{np_0}}{\log n}.
		\end{align*}
	\end{theorem}
	
	We also need  upper bound on the minimum degree of a random bipartite graph.
	The following lemma is a version of Lemma 2.2 in~\cite{krivelevich2012optimal}, stated here for random bipartite graphs and for a larger regime of $p$, but the proof is essentially similar to the original statement.
	\begin{lemma}[\cite{krivelevich2012optimal}]
		\label{lem:deltaUppr}
		If $\frac{\log n}{n} \le p_0 < 1$ then with high probability
		\[np_0 - 2\sqrt{np\log n} \le \delta(B_{n,p_0}) \le np_0 - \frac{1}{3}\sqrt{np_0 \log n}. \]
	\end{lemma}

	\hspace{10pt}
	
	Let $\frac{\log^{15} n}{n} \le p < \varepsilon$, where we always have $\varepsilon = \frac{8^8}{2(9e)^9}$ throughout the paper, and let $p_0, p_1$ be satisfying $(1-p_0)(1-p_1) = 1-p$, where $p_1 = \sqrt{\frac{p}{n \log^4 n}}$.
	We show that typically $\delta^{\pm}\left(D_{n,p} \right) = \delta$, that is, that the number of $1$-factors found in step $4$ of the procedure is precisely the quantity we aim for.
	\begin{lemma}
		\label{lem:mindegDnp}
		Let $D'$ be a digraphs generated by Procedure~\ref{procedure:D}.
		Then with high probability we have $\delta^{\pm}\left(D_{n,p} \right) = \delta^{\pm}\left(D' \right)$.
	\end{lemma}
	
	\begin{proof}
		Recall that we generate $D'$ in few steps: First we expose a random bipartite graph $B'$ with edge-probability $p_0$, then we we extend it to a bipartite graph $B$ by considering the minimum degree vertices in each part of $B'$, and exposing all the edges incident to them with probability $p_1$, and lastly, we take a random permutation $\pi$ to generate $D'$, as explained in the procedure.
		Observe that, with high probability, the last part does not affect the minimum degree.
		Therefore, it is enough to show that with high probability $\delta(B) = \delta^{\pm}\left(D_{n,p} \right)$.
		
		By \Cref{thm:DegGap}, with high probability there is a unique vertex of minimum degree in $B'$, denote it by $x$, and the second minimum degree is $\delta_2(B') \ge \delta(B') + \frac{\sqrt{np_0}}{\log n}$.
		Then with high probability, after exposing edges from $x$ with probability $p_1$, we have $\deg_B(x) = \delta(B') + O(np_1)$.
		If we have $\frac{\sqrt{np_0}}{\log n} = \omega\left(np_1 \right)$ then $\deg_B(x) \le \delta_2(B') \le \delta_2(B)$ which means that $x$ has minimum degree in $B$ as well.
		Indeed, since $p_0 = \frac{p - p_1}{1 - p_1} \ge p - p_1 \ge \frac{p}{2}$, we get that
		\[\frac{\sqrt{np_0}}{\log n} \ge \frac{\sqrt{np}}{2\log n} = \omega\left(\frac{\sqrt{np}}{\log^2 n} \right) = \omega(np_1). \]
		So we have $\delta^{\pm}(D') = \delta(B) = \deg_B(x)$.
		
		Moreover, with high probability, for any $v \in [n]$ we have
		\begin{align*}
		\deg^{\pm}_{D_{n,p}}(x) = \deg^{\pm}_{D'}(x) \le \deg^{\pm}_{D'}(v) \le \deg^{\pm}_{D_{n,p}}(v),
		\end{align*}
		where $\deg^{\pm}_D(v) \coloneqq \min \{\deg^-_D(v), \deg^+_D(v)\}$.
		Thus, with high probability we have $\delta^{\pm}\left(D_{n,p} \right) = \deg^{\pm}_{D_{n,p}}(x) = \deg^{\pm}_{D'}(x) = \delta^{\pm}(D')$.
	\end{proof}

	It is now left to show that typically step $4$ of Procedure~\ref{procedure:D} does not output FAILURE.

	\begin{lemma}
		\label{lem:disjointPM}
		Let $B$ be a random bipartite graph generated as in the description given in Steps 1-3 of Procedure~\ref{procedure:D}, and let $\delta \coloneqq \min\left\{\deg_{B}\left(x^+ \right), \deg_{B}\left(y^- \right) \right\}$.
		Then with high probability there are $\delta$ edge-disjoint perfect matchings in $B$.
	\end{lemma}
	
	\begin{proof}
		Firstly, note that the graph generated at Step $1$ of Procedure~\ref{procedure:D} is the random binomial bipartite graph $B_{n,p_0}$.
		Using \Cref{lem:GaleRyser} it is enough to prove the following claim.
		\begin{claim}
			\label{cl:step4lemma}
			With high probability, for any $X' \subseteq X$ and $Y' \subseteq Y$ we have
			\begin{align}
			\label{eq:step4claim}
			e_{B}\left(X', Y' \right) \ge \delta\left(|X'| + |Y'| - n \right).
			\end{align}
		\end{claim}
		
		\begin{proof}
			Trivially, (\ref{eq:step4claim}) holds for any $X', Y'$ with $|X'| + |Y'| \le n$.
			Hence, let $X' \subseteq X$ and $Y' \subseteq Y$ be such that $|X'| + |Y'| \ge n+1$.
			Denote $x\coloneqq |X'|$, $y\coloneqq |Y'|$, and assume, without loss of generality, that $x \le y$.
			Let $\delta_2 \coloneqq \delta_2(B)$ and $\gamma \coloneqq \delta_2 - \delta$.
			Note that $X'$ might contain the min-degree vertex, but yet, we always have $e_B(X', Y) \ge \delta_2(x-1) + \delta$.
			Let $Z' \coloneqq Y \setminus Y'$ and $z \coloneqq |Z'|$, and note that since $x+y\geq n+1$ we have that $z \le x-1$. Assume towards a contradiction that $e_B(X',Y')< \delta (x+y-n)$. This implies that
			\begin{align}
			\begin{split}
			\label{eq:eXZ}
			e_B(X', Z') &= e_B(X', Y) - e_B(X', Y') \\
			&> (\delta+\gamma)(x-1) + \delta - \delta(x-z) \\
			&= \gamma(x-1) + \delta z.
			\end{split}
			\end{align}
			We show that with high probability, for every $X'$ and $Z'$ as above, the inequality (\ref{eq:eXZ}) does not hold.
			
			\paragraph*{Case 1:} $x \leq \delta_2$. Since $z\leq x-1$ and $x-\delta\leq \delta_2-\delta=\gamma$, we have that 
			$z(x-\delta)\leq \gamma (x-1)$.
			By rearranging we obtain that
			$e_B(X', Z') \leq xz\leq \gamma(x-1) + \delta z$, which contradicts (\ref{eq:eXZ}).

			\paragraph*{Case 2:} 
			$\delta_2 < x  < \frac{n}{2e}$.
			Note that in this case $x = \omega(\gamma)$ and therefore we have that $\gamma(x-1)=(1-o(1))\gamma x$.
			Using the union bound, the probability of having sets $X'$ and $Z'$ for which (\ref{eq:eXZ}) holds is at most 
			\begin{align*}
			\sum_{\delta_2 < x < \frac{n}{2e}} \sum_{z\le x} \binom{n}{x} \binom{n}{z} \binom{xz}{\gamma x + \delta z} p^{\gamma x + \delta z} &\le \sum_{\delta_2 < x < \frac{n}{2e}} \left(\frac{en}{x} \right)^{2x} \sum_{z \le x} \left(\frac{exzp}{\gamma x + \delta z} \right)^{\gamma x + \delta z} \\
			&\le \sum_{\delta_2 < x < \frac{n}{2e}} \left(\frac{en}{x} \right)^{2x} \sum_{z \le x} \left(\frac{exp}{\delta} \right)^{\gamma x} \\
			&\le \sum_{\delta_2 < x < \frac{n}{2e}} \left(\frac{en}{x} \right)^{2x} \sum_{z \le x} \left(\frac{1.5ex}{n} \right)^{\gamma x}
			&=o(1),		
			\end{align*}
			where the last inequality holds since $\delta\geq \frac{2np}{3}$ and the last equality holds since $\gamma=\omega(1)$.

			\paragraph*{Case 3:}
			$\frac{n}{2e} \le x$ and $z < \frac{\gamma}{3p}$.
			For given $X'$ and $Z'$ we have
			\begin{align*}
			\Pr\left[e_B(X',Z') > \gamma x + \delta z \right] \le \Pr\left[e_B(X',Z') > \gamma x \right] \le \binom{xz}{\gamma x}p^{\gamma x} \le \left(\frac{ezp}{\gamma} \right)^{\gamma x} < \left(\frac{e}{3} \right)^{\gamma x}.
			\end{align*}
			Recall that $\gamma = \omega(1)$ and $x\geq n/2e$ (so in particular we have $\gamma x=\omega(n)$) and therefore, by applying the union bound we obtain that the probability for such $X',Z'$ to exist is at most
			$$2^n\cdot  2^n\cdot \left(\frac{e}{3}\right)^{\omega(n)}=o(1).$$

			\paragraph*{Case 4:}
			$\frac{n}{2e} \le x < n-\frac{n}{\log \log n}$ and $z\geq \frac{\gamma}{3p}= \omega\left(\frac{1}{p} \right)$.
			By \Cref{lem:deltaUppr} we have $\delta= np - 2\sqrt{np\log n} \ge np-\frac{np}{\log^2\log n}$.
			In particular, for $\eta=(\log\log n)^{-1}$ we have $(1+\eta)xp < \delta$, and hence $(1+\eta)xzp < \delta z$.
			Moreover, by combining \Cref{lem:deltaUppr} and \Cref{thm:DegGap} we get $\gamma \ge \frac{\sqrt{np_0}}{2\log n} \ge \log^6 n$, so $\gamma \eta^2=\omega(1)$.
			Therefore, by \Cref{Chernoff}, for given $X', Z'$ we have that
			\begin{align*}
			\Pr\left[e_B(X',Z') > \gamma x + \delta z \right] \le \Pr\left[e_B(X',Z') > (1+\eta)xzp \right] \le e^{-\frac{\eta^2}{2+\eta}xzp}=2^{-\omega(n)}.
			\end{align*}
			Taking a union bound over all possible subsets, the probability for such $X',Z'$ to exist is at most
			$$2^n\cdot  2^n\cdot 2^{-\omega(n)}=o(1).$$
			
			\paragraph*{Case 5:}
			$x\geq n-\frac{n}{\log\log n}$.
			Recall that $y\ge x$, so we also have $y \ge n - \frac{n}{\log\log n}$.
			Note that
			\begin{align*}
			\frac{\delta}{n}xy - \delta(x+y-n) = \frac{\delta}{n}(n-x)(n-y) \ge 0,
			\end{align*}
			hence it is sufficient to show that with high probability we have $e_B(X',Y') \ge \frac{\delta}{n}xy$ for all such $X', Y'$.
			By \Cref{lem:deltaUppr} we know that with high probability we have $\delta \le np - \frac{1}{3}\sqrt{np\log n}$.
			We write
			\begin{align*}
			\frac{\delta}{n}xy = xyp(1-\eta),
			\end{align*}
			where
			\begin{align*}
			\eta = \frac{np - \delta}{np} \ge \frac{1}{3}\sqrt{\frac{\log n}{np}} > 0.
			\end{align*}
			Given $X',Y'$ as above, by \Cref{Chernoff} we get that
			\begin{align*}
			\Pr\left[e_B(X',Y') < \frac{\delta}{n}xy \right] &= \Pr\left[e_B(X',Y') < xyp(1-\eta) \right] \le e^{-\frac{\eta^2}{2}xyp}=2^{-\omega(n)}.
			\end{align*}
			Taking a union bound over all possible subsets, we obtain that the probability for such $X',Y'$ to exist is at most $2^n\cdot 2^n\cdot 2^{-\omega(n)}=o(1)$. This completes the proof of the claim.
		\end{proof}
		
		By applying \Cref{cl:step4lemma} and \Cref{lem:GaleRyser} to $B$ we conclude that with high probability $B$ contains a $\delta$-factor, as desired. This completes the proof of \Cref{lem:disjointPM}.
	\end{proof}
	
	As $B$ typically contains $\delta$ edge-disjoint perfect matchings, the digraph $D' \coloneqq D_{\pi}(B)$ contains $\delta$ edge-disjoint $1$-factors, for any choice of permutation $\pi \in S_n$.
	However, as we have already mentioned, we would like these $1$-factors in $D'$ to satisfy some requirements.
	The first one is to consist of ``not too many'' cycles. Since the expected number of cycles in a random permutation is $O(\log n)$, it makes sense to choose $\pi \in S_n$ uniformly at random.
	
	\begin{lemma}
		\label{lem:fewcycles}
		Let $B=(X\cup Y,E)$ be a bipartite graph with $|X|=|Y|=n$, and suppose that $B$ contains $r\leq n$ edge-disjoint perfect matchings.
		Let $\pi \in S_n$ be a permutation chosen uniformly at random, and consider the digraph $D' \coloneqq D_{\pi}(B)$. Then $D':=D_{\pi}(B)$ contains $r$ edge-disjoint $1$-factors and with high probability each of which consists of at most $4\log n$ directed cycles.
	\end{lemma}
	
	\begin{proof}
		Let $\mathcal M$ be a collection of $r \le n$ edge-disjoint perfect matchings in $B$, and observe that every $M\in \mathcal M$ is translated into a $1$-factor in $D'$.
		Moreover, for every $M\neq M' \in \mathcal M$, since $M$ and $M'$ are edge-disjoint, it follows that they are being translated to edge-disjoint $1$-factors in $D'$.
		
		Next, observe that for a given $M\in \mathcal M$, by choosing $\pi\in S_n$ uniformly at random we have that the $1$-factor that $M$ is being mapped to in $D'$ corresponds to the cyclic form of a uniformly chosen permutation.
		Denote by $\sigma(\pi)$ and $\sigma(M)$ the numbers of cycles in $\pi$ and in $M$, repsectively.
		We use the simple and very nice fact from~\cite{ford2021cycle} stating that $\mathbb E\left[2^{\sigma(\pi)} \right] = n+1$.
		Therefore, by Markov inequality, we have that
		\[\Pr\left[\sigma(M) \ge 4\log n \right] \le \Pr\left[2^{\sigma(\pi)} \ge (n+1)^3 \right] \le \frac{1}{(n+1)^2}. \]
		Taking a union bound over all $M\in \mathcal M$ concludes the proof.
	\end{proof}
	
	By combining \Cref{lem:mindegDnp,lem:disjointPM,lem:fewcycles} we get that with high probability the digraph $D'$ obtained from Procedure~\ref{procedure:D} contains $\delta = \delta^{\pm}(D_{n,p})$ many edge-disjoint $1$-factors.
	Moreover, there exists a collection $\mathcal F$ of $\delta$ edge-disjoint $1$-factors in $D'$ such that each of which has at most $4\log n$ cycles.
	
	The following is an immediate corollary of \Cref{lem:fewcycles}, given a simple pigeonhole argument.
	\begin{corollary}
		\label{lem:longcycle}
		With high probability, each $1$-factor $F\in \mathcal F$ contains a cycle of length at least $\frac{n}{4\log n}$.
	\end{corollary}

	Before moving to the next phase, we need one more property of the digraph obtained at the end of Phase $1$, in order to avoid exposing those edges which already appear in the graph $D'$.
	Thus, another important property of our digraph $D'$ consists of the following definition.
	\begin{definition}
		\label{def:heavyvtx}
		Given a digraph $D$ on $n$ vertices, a subset $X \subset V(D)$ and a real number $c \in [0,1]$, we say that a vertex $v \in X$ is \emph{$(X,c)$-heavy} if it has either at least $c|X|$ many out-neighbours or at least $c|X|$ many in-neighbours in $X$.
	\end{definition}
	We will use the following simple proposition.
    \begin{proposition}
	\label{prop:heavyvtx}
		Let $B = (X \cup Y, E)$ be a bipartite graph with $|X| = |Y| = n$ and with maximum degree at most $2np$ for some $0 \le p \le \varepsilon$, and suppose that $B$ contains a collection $\cM$ of $r \le n$ edge-disjoint perfect matchings.
		Let $\pi \in S_n$ be a permutation chosen uniformly at random, and consider $D' \coloneqq D_{\pi}(B)$.
		Then $D'$ contains a collection $\cF$ of $r$ edge-disjoint $1$-factors and with high probability contains no $\left(C, \frac{1}{9} \right)$-heavy vertices for all cycles $C$ of length at least $\frac{n}{\log^3 n}$ in all $1$-factors in $\cF$.
	\end{proposition}

    \begin{proof}
		Let $\cM = \{M_1,\ldots,M_r\}$ be the collection of $r$ edge-disjoint perfect matchings in $B$, and let $v \in X$.
        Let $\cF = \{F_1, \ldots, F_r \}$ be the collection of $r$ edge-disjoint $1$-factors in $D'$ such that for each $i \in [r]$ the perfect matching $M_i$ was mapped to the $1$-factor $F_i$ under $\pi$.
        Given $i\in [r]$, we want to bound from above the probability that $F_i$ contains a cycle $C$ of length $|C|=t\ge \frac{n}{\log^3 n}$, such that $v$ has at least $\frac{1}{9}t$ neighbors in $C$.
        We will then take a union bound over all vertices and all matchings (or $1$-factors).

        Fix $t\ge \frac{n}{\log^3 n}$.
        There are at most $\binom{\deg(v)}{t/9} \le \binom{2np}{t/9}$ many ways to choose $\frac{1}{9}t$ neighbours of $v$ in $Y$, and at most $\binom{n}{8t/9}$ many ways to choose the remaining vertices in $C$.
        Let $T_i$ denote this set of $t$ vertices chosen in $Y$, and let $S_i \subseteq X$ be their neighbors in $M_i$.
        The probability that $T_i\cup S_i$ is mapped into a cycle under $\pi$ is at most $\frac{(t-1)!}{\binom{n}{t}t!}=\frac{1}{t\binom{n}{t}}$. Therefore, the probability that $v$ is $(C,\frac 19)$-heavy for some $C$ of length $t$ in $F_i$ is at most
        \begin{align*}
        	\frac{\binom{2np}{t/9} \binom{n}{8t/9}}{t\binom{n}{t}} \le \frac{\left(\frac{2enp}{t/9} \right)^{\frac{1}{9}t} \left(\frac{en}{8t/9} \right)^{\frac{8}{9}t}}{t \left(\frac{n}{t} \right)^t} = \frac{1}{t} \left(\frac{2(9e)^9}{8^8} \cdot p \right)^{\frac{1}{9}t} = o\left(\frac{1}{n^3} \right),
        \end{align*}
        where the last inequality follows as $t \ge \frac{n}{\log^3 n}$.
        Taking a union bound over all $\frac{n}{\log^3 n} \le t \le n$, and over all vertices and matchings, completes the proof.
	\end{proof}

	
	\section{\texorpdfstring{Phase 2: Generating $D''$ while converting $1$-factors into Hamilton cycles}{Phase 2: Generating D'' while converting 1-factors into Hamilton cycles}}
	
	Here we formally describe Phase 2 from the proof outline. 
	Our goal is to convert the family of (edge-disjoint) $1$-factors from the previous section into a family of edge-disjoint Hamilton cycles using edges that we expose in an online manner.
	Given a probability parameter $q$, when we say that an edge is being \emph{exposed} we mean that we perform a Bernoulli trial with probability $q$ for success to decide whether this edge is going to be added to our graph or not.
	When we say that an edge has been \emph{successfully exposed} we mean that the corresponding Bernoulli trial ended up as a ``success'' and the edge has been added as an edge to our graph.
	All the Bernoulli trials are independent.
	It is important to note that we may expose the same edge several times, and therefore, in order to show that the digraph that we end up with can be coupled as a subgraph of $D_{n,p_1}$, we will show that with high probability no edge has been exposed ``too many'' times.
	
	The initial set of edges that we expose is the set of available edges of $D'$ at the end of Phase $1$,
	\[E' \coloneqq \left\{\overrightarrow{uv} ~:~ u \in [n]\setminus \{x^+ \}, \; v \in [n]\setminus \{\pi(y^-) \} \right\} \setminus E(D'). \]
	Throughout Phase $2$ we will update the set $E'$ from which we expose edges.
	Given a vertex $u \in V$ and subsets of vertices $A, B \subset V$ and a set of available edges $E'$, we write $E'(u,A) \coloneqq E(u,A) \cap E'$, $E'(A,u) \coloneqq E(A,u) \cap E'$ and $E'(A,B) = E(A,B) \cap E'$.

	\subsection{Path rotations and online sprinkling}	
	Here we give an exact description of how to perform what we call ``double rotations'' to a given path $P$, using (currently) available edges that we expose in an online fashion (that is, exposing edges only when needed).
	During this exposure we allow an edge to be exposed multiple times, and at the end we show that it is in fact unlikely that we successfully expose an edge more than once.
	
	Let $P = (u_1, \ldots, u_m)$ be a directed path on $m$ vertices, for some integer $m \ge 5$.
	Inspired by the \emph{double-rotation} argument that was introduced by Frieze~\cite{frieze1988algorithm}, we now describe how to rotate a directed path $P$ from each of its endpoints separately, to obtain a new path.
	Define
	\begin{align*}
	V_1 &\coloneqq V_1(P) = \left\{u_i ~:~ 1\le i < \frac{m}{4} \right\} \\
	V_2 &\coloneqq V_2(P) = \left\{u_j ~:~ \frac{m}{4} \le j < \frac{m}{2} \right\} \\
	V_3 &\coloneqq V_3(P) = \left\{u_s ~:~ \frac{m}{2} < s \le \frac{3m}{4} \right\} \\
	V_4 &\coloneqq V_4(P) = \left\{u_t ~:~ \frac{3m}{4} < t \le m \right\}.
	\end{align*}
	
	\paragraph*{Left-rotation of $P$.}
	Given $x\in V_1$ and $y\in V_2$, define $P_{\ell} \coloneqq \left(P\setminus \left\{\overrightarrow{xx_+},\overrightarrow{yy_+}\right\} \right)\cup\left\{\overrightarrow{yu_1},\overrightarrow{xy_+}\right\}$.
	We say that $P_{\ell}$ is obtained from $P$ by a \emph{double left-rotation} (or \emph{left-rotation} for short) with pivots $(y,x)$. 
	Observe that $P_{\ell}$ is a path (see \Cref{fig:pathrotation}) such that $V(P_{\ell})=V(P)$, $x_+$ is its left endpoint, and $V_1(P_{\ell})\cup V_2(P_{\ell})=V_1(P)\cup V_2(P)$.
	Moreover, performing a left-rotation to $P$ does not change its right half.
	
	\paragraph*{Right-rotation of $P$.}
	Given $z \in V_4$ and $w \in V_3$, define $P_r \coloneqq \left(P \setminus \left\{\overrightarrow{z_-z}, \overrightarrow{w_-w}\right\} \right) \cup \left\{\overrightarrow{u_mw},\overrightarrow{w_-z} \right\}$.
	We say that $P_r$ is obtained from $P$ by a \emph{double right-rotation} (or \emph{left-rotation} for short) with pivots $(z,w)$.
	Observe that $P_r$ is a path (see \Cref{fig:pathrotation}) such that $V(P_r) = V(P)$, $z_-$ is its right endpoint, and $V_3(P_r) \cup V_4(P_r) = V_3(P) \cup V_4(P)$.
	Moreover, performing a right-rotation to $P$ does not change its left half.
	
	\begin{remark}
		We remark that left- and right-rotations can also be defined for vertices of $P$ which are not necessarily chosen from a specific $V_i$.
		However, for technical reasons which will be explained below, considering only vertices from $V_1, V_2$ and from $V_3, V_4$ for left- and right-rotations, respectively, makes our proof simpler.
	\end{remark}

	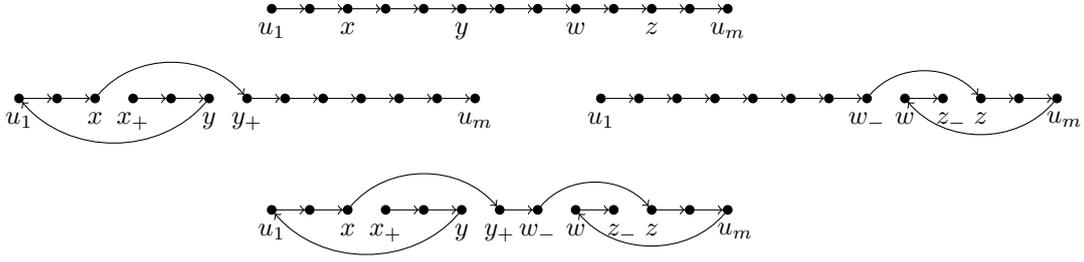
\begin{figure}
		\centering
		\begin{subfigure}[b]{0.5\textwidth}
			\begin{tikzpicture}[auto, vertex/.style={circle,draw=black!100,fill=black!100, thick,inner sep=0pt,minimum size=1mm}]
			\node (u1) at (-3,0) [vertex,label=below:$u_1$] {};
			\node (u2) at (-2.5,0) [vertex] {};
			\node (x) at (-2,0) [vertex,label=below:$x$] {};
			\node (x+) at (-1.5,0) [vertex] {};
			\node (u5) at (-1,0) [vertex] {};
			\node (y) at (-0.5,0) [vertex,label=below:$y$] {};
			\node (y+) at (0,0) [vertex] {};
			\node (w-) at (0.5,0) [vertex] {};
			\node (w) at (1,0) [vertex,label=below:$w$] {};
			\node (z-) at (1.5,0) [vertex] {};
			\node (z) at (2,0) [vertex,label=below:$z$] {};
			\node (u12) at (2.5,0) [vertex] {};
			\node (um) at (3,0) [vertex,label=below:$u_m$] {};
			\node[text width=1cm] at (0.25,1.5) {$P$};
			
			\draw [->] (u1) --node[inner sep=0pt,swap]{} (u2);
			\draw [->] (u2) --node[inner sep=2pt,swap]{} (x);
			\draw [->] (x) --node[inner sep=2pt,swap]{} (x+);
			\draw [->] (x+) --node[inner sep=2pt,swap]{} (u5);
			\draw [->] (u5) --node[inner sep=2pt,swap]{} (y);
			\draw [->] (y) --node[inner sep=2pt,swap]{} (y+);
			\draw [->] (y+) --node[inner sep=2pt,swap]{} (w-);
			\draw [->] (w-) --node[inner sep=2pt,swap]{} (w);
			\draw [->] (w) --node[inner sep=2pt,swap]{} (z-);
			\draw [->] (z-) --node[inner sep=2pt,swap]{} (z);
			\draw [->] (z) --node[inner sep=2pt,swap]{} (u12);
			\draw [->] (u12) --node[inner sep=2pt,swap]{} (um);
			\end{tikzpicture}
			\label{subfig:norotation}
		\end{subfigure}
		\begin{subfigure}[b]{1\textwidth}
			\begin{minipage}{.5\textwidth}
				\centering
				\begin{tikzpicture}[auto, vertex/.style={circle,draw=black!100,fill=black!100, thick,inner sep=0pt,minimum size=1mm}]
				\node (u1) at (-3,0) [vertex,label=below:$u_1$] {};
				\node (u2) at (-2.5,0) [vertex] {};
				\node (x) at (-2,0) [vertex,label=below:$x$] {};
				\node (x+) at (-1.5,0) [vertex,label=below:$x_+$] {};
				\node (u5) at (-1,0) [vertex] {};
				\node (y) at (-0.5,0) [vertex,label=below:$y$] {};
				\node (y+) at (0,0) [vertex,label=below:$y_+$] {};
				\node (w-) at (0.5,0) [vertex] {};
				\node (w) at (1,0) [vertex] {};
				\node (z-) at (1.5,0) [vertex] {};
				\node (z) at (2,0) [vertex] {};
				\node (u12) at (2.5,0) [vertex] {};
				\node (um) at (3,0) [vertex,label=below:$u_m$] {};

				\draw [->] (u1) --node[inner sep=0pt,swap]{} (u2);
				\draw [->] (u2) --node[inner sep=2pt,swap]{} (x);
				\draw [->] (x+) --node[inner sep=2pt,swap]{} (u5);
				\draw [->] (u5) --node[inner sep=2pt,swap]{} (y);
				\draw [->] (y+) --node[inner sep=2pt,swap]{} (w-);
				\draw [->] (w-) --node[inner sep=2pt,swap]{} (w);
				\draw [->] (w) --node[inner sep=2pt,swap]{} (z-);
				\draw [->] (z-) --node[inner sep=2pt,swap]{} (z);
				\draw [->] (z) --node[inner sep=2pt,swap]{} (u12);
				\draw [->] (u12) --node[inner sep=2pt,swap]{} (um);
				\draw [->]   (y) to[out=-130,in=-50] (u1);
				\draw [->]   (x) to[out=50,in=130] (y+);
				\end{tikzpicture}
			\end{minipage}%
			\begin{minipage}{.5\textwidth}
				\centering
				\begin{tikzpicture}[auto, vertex/.style={circle,draw=black!100,fill=black!100, thick,inner sep=0pt,minimum size=1mm}]
				\node (u1) at (-3,0) [vertex,label=below:$u_1$] {};
				\node (u2) at (-2.5,0) [vertex] {};
				\node (x) at (-2,0) [vertex] {};
				\node (x+) at (-1.5,0) [vertex] {};
				\node (u5) at (-1,0) [vertex] {};
				\node (y) at (-0.5,0) [vertex] {};
				\node (y+) at (0,0) [vertex] {};
				\node (w-) at (0.5,0) [vertex,label=below:$w_-$] {};
				\node (w) at (1,0) [vertex,label=below:$w$] {};
				\node (z-) at (1.5,0) [vertex,label={[shift={(0.11,-0.61)}]$z_-$}] {};
				\node (z) at (2,0) [vertex,label=below:$z$] {};
				\node (u12) at (2.5,0) [vertex] {};
				\node (um) at (3,0) [vertex,label={[shift={(0.1,-0.57)}]$u_m$}] {};

				\draw [->] (u1) --node[inner sep=0pt,swap]{} (u2);
				\draw [->] (u2) --node[inner sep=2pt,swap]{} (x);
				\draw [->] (x) --node[inner sep=2pt,swap]{} (x+);
				\draw [->] (x+) --node[inner sep=2pt,swap]{} (u5);
				\draw [->] (u5) --node[inner sep=2pt,swap]{} (y);
				\draw [->] (y) --node[inner sep=2pt,swap]{} (y+);
				\draw [->] (y+) --node[inner sep=2pt,swap]{} (w-);
				\draw [->] (w) --node[inner sep=2pt,swap]{} (z-);
				\draw [->] (z) --node[inner sep=2pt,swap]{} (u12);
				\draw [->] (u12) --node[inner sep=2pt,swap]{} (um);
				\draw [->]   (um) to[out=-130,in=-50] (w);
				\draw [->]   (w-) to[out=50,in=130] (z);
				\end{tikzpicture}
			\end{minipage}
			\label{subfig:leftrightrotation}
		\end{subfigure}
		\begin{subfigure}[b]{0.5\textwidth}
			\begin{tikzpicture}[auto, vertex/.style={circle,draw=black!100,fill=black!100, thick,inner sep=0pt,minimum size=1mm}]
			\node (u1) at (-3,0) [vertex,label=below:$u_1$] {};
			\node (u2) at (-2.5,0) [vertex] {};
			\node (x) at (-2,0) [vertex,label=below:$x$] {};
			\node (x+) at (-1.5,0) [vertex,label=below:$x_+$] {};
			\node (u5) at (-1,0) [vertex] {};
			\node (y) at (-0.5,0) [vertex,label=below:$y$] {};
			\node (y+) at (0,0) [vertex,label=below:$y_+$] {};
			\node (w-) at (0.5,0) [vertex,label=below:$w_-$] {};
			\node (w) at (1,0) [vertex,label=below:$w$] {};
			\node (z-) at (1.5,0) [vertex,label={[shift={(0.11,-0.61)}]$z_-$}] {};
			\node (z) at (2,0) [vertex,label=below:$z$] {};
			\node (u12) at (2.5,0) [vertex] {};
			\node (um) at (3,0) [vertex,label={[shift={(0.1,-0.57)}]$u_m$}] {};

			\draw [->] (u1) --node[inner sep=0pt,swap]{} (u2);
			\draw [->] (u2) --node[inner sep=2pt,swap]{} (x);
			\draw [->] (x+) --node[inner sep=2pt,swap]{} (u5);
			\draw [->] (u5) --node[inner sep=2pt,swap]{} (y);
			\draw [->] (y+) --node[inner sep=2pt,swap]{} (w-);
			\draw [->] (w) --node[inner sep=2pt,swap]{} (z-);
			\draw [->] (z) --node[inner sep=2pt,swap]{} (u12);
			\draw [->] (u12) --node[inner sep=2pt,swap]{} (um);
			\draw [->]   (um) to[out=-130,in=-50] (w);
			\draw [->]   (y) to[out=-130,in=-50] (u1);
			\draw [->]   (w-) to[out=50,in=130] (z);
			\draw [->]   (x) to[out=50,in=130] (y+);
			\end{tikzpicture}
			\label{subfig:doublerotation}
		\end{subfigure}
		\caption{A path $P$, left- and right-rotations of it with pivots $(x,y)$ and $(w,z)$, respectively, and a path obtained by rotations from both sides.}
		\label{fig:pathrotation}
	\end{figure}

	\paragraph*{Online sprinkling left rotations.}
	Given a path $P = (u_1, \ldots, u_m)$ in a digraph $D'$ on $n$ vertices and a subset of available edges $E'$, do the following.
	\begin{enumerate}
		\item Expose all edges in the set $E'\left(V_2, u_1 \right)$ independently with probability $\frac{100\log n}{m}$.
		Let $\IN{u_1}{V_2}$ be the in-neighbourhood of $u_1$ in $V_2$ in the digraph spanned by the successfully exposed edges.
		\item For each $y \in \IN{u_1}{V_2}$ expose all edges in the set $E'\left(V_1, y_+ \right)$ independently with probability $\frac{100\log n}{m}$.
		Denote by $\IN{y_+}{V_1}$ the in-neighbourhood of $y_+$ in $V_1$ in the digraph spanned by the successfully exposed edges.
	\end{enumerate}
	For each $y \in \IN{u_1}{V_2}$ and $x \in \IN{y_+}{V_1}$ apply left-rotation to $P$ with pivots $(x,y)$, and let $\cP_{\ell}^1 \coloneqq \cP_{\ell}^1(P)$ be the family of obtained paths.
	Moreover, denote by
	\[\END_{\ell}^1 \coloneqq \END_{\ell}^1(P) = \bigcup_{y \in \IN{u_1}{V_2}} \left(\IN{y_+}{V_1} \right)_+ \]
	the set of left endpoints of paths in $\cP_{\ell}^1$.
	Observe that for every element in $\END_{\ell}^1$ there exists a path in $\cP_{\ell}^1$ for which it is its left endpoint.
	
	\paragraph*{Online sprinkling right rotations.}
	Given a path $P = (u_1, \ldots, u_m)$ in a digraph $D'$ and a subset of available edges $E'$, we do the following.
	\begin{enumerate}
		\item Expose all edges in the set $E'\left(u_m, V_3 \right)$ independently with probability $\frac{100\log n}{m}$.
		Let $\OUT{u_m}{V_3}$ be the out-neighbourhood of $u_m$ in $V_3$ in the digraph spanned by the successfully exposed edges.
		\item For each $w \in \OUT{u_m}{V_3}$ expose all edges in the set $E'\left(w_-,V_4 \right)$ independently with probability $\frac{100\log n}{m}$.
		Denote by $\OUT{w_-}{V_4}$ the out-neighbourhood of $w_-$ in $V_4$ in the digraph spanned by the successfully exposed edges.
	\end{enumerate}
	For each $w \in \OUT{u_m}{V_3}$ and $z \in \OUT{w_-}{V_4}$ apply right-rotation to $P$ with pivots $(w,z)$, and let $\cP_r^1 \coloneqq \cP_r^1(P)$ be the family of obtained paths.
	Moreover, denote by
	\[\END_r^1 \coloneqq \END_r^1(P) = \bigcup_{w \in \OUT{u_m}{V_3}} \left(\OUT{w_-}{V_4} \right)_- \]
	the set of right endpoints of paths in $\cP_r^1$.
	Observe that for every element in $\END_r^1$ there exists a path in $\cP_r^1$ for which it is its right endpoint.

	\paragraph*{Rotating $P$ from both sides using online sprinkling.}
	In a digraph $D'$, let $P = (u_1, \ldots, u_m)$ be a path, $E'$ be a subset of available edges 
	and $(t_{\ell}, t_r)$ be a tuple of positive integers.
	Denote $\END_{\ell}^0 \coloneqq \{u_1\}$, $\END_r^0 \coloneqq \{u_m\}$ and $\mathcal P_{\ell}^0 = \mathcal P_r^0 \coloneqq \{P\}$.
	For each $1 \le t'_{\ell} \le t_{\ell}$ and for each $1 \le t'_r \le t_r$ we do the following.
	\begin{enumerate}
		\item \textbf{Left rotations:} For each $u \in \END_{\ell}^{t'_{\ell}-1}$ let $P^u \in \mathcal P_{\ell}^{t'_{\ell}-1}$ be a path for which $u$ is its left endpoint.
		Perform online sprinkling left rotations to $P^u$ with $E'$ to obtain the sets $\mathcal P_{\ell}^1(P^u)$ and $\END_{\ell}^1(P^u)$ of all obtained paths and their corresponding left endpoints, respectively.
		\item If $\bigcup_{u\in \mathcal \END_{\ell}^{t'_{\ell}-1}} \mathcal P_{\ell}^1(P^u) = \emptyset$ let $\mathcal P_{\ell}^{t'_{\ell}} = \mathcal P_{\ell}^{t'_{\ell}-1}$ and $\END_{\ell}^{t'_{\ell}} = \END_{\ell}^{t'_{\ell}-1}$.
		Otherwise, let $\mathcal P^{t'_{\ell}}_{\ell} \coloneqq \bigcup_{u \in \mathcal \END^{t'_{\ell}-1}_{\ell}} \mathcal P_{\ell}^1(P^u)$, and $\END_{\ell}^{t'_{\ell}} \coloneqq \bigcup_{u \in \mathcal \END^{t'_{\ell}-1}_{\ell}} \END_{\ell}^1(P^u)$.
		\item \textbf{Right rotations:} For each $u \in \END_r^{t'_r-1}$ let $P^u \in \mathcal P_r^{t'_r-1}$ be a path for which $u$ is its right endpoint.
		Perform online sprinkling right rotations to $P^u$ with $E'$ to obtain the sets $\mathcal P_r^1(P^u)$ and $\END_r^1(P^u)$ of all obtained paths and their corresponding right endpoints, respectively.
		\item If $\bigcup_{u \in \END_r^{t'_r-1}} \mathcal P_r^1(P^u) = \emptyset$ let $\mathcal P_r^{t'_r} = \mathcal P_r^{t'_r-1}$ and $\END_r^{t'_r} = \END_r^{t'_r-1}$.
		Otherwise, let $\mathcal P_r^{t'_r} \coloneqq \bigcup_{u\in \END_r^{t'_r-1}} \mathcal P_r^1(P^u)$, and $\END_r^{t'_r} \coloneqq \bigcup_{u \in \END_r^{t'_r-1}} \END_r^1(P^u)$.
	\end{enumerate}
	After performing $t_{\ell}$ and $t_r$ iterations of left and right rotations,  respectively, with online sprinkling, we obtain the families of paths $\mathcal P_{\ell}^{t_{\ell}}, \mathcal P_r^{t_r}$ and the sets of their left and right endpoints $\END_{\ell}^{t_{\ell}}$ and $\END_r^{t_r}$, respectively.
	
	Observe that for every $x\in \END_{\ell}^{t_{\ell}}$ and every $y\in \END_r^{t_r} $ there exists a path $P'$ such that $x$ and $y$ are its left and right endpoints, respectively, and which is obtained from $P$ by performing left- anf right-rotations.
	Therefore, any edge in $\END_r^{t_r}\times \END_{\ell}^{t_{\ell}}$ completes $P'$ into a cycle on the vertex set $V(P)$.
	In order to show that such an edge typically exists when we expose edges in this set, it would be convenient for us to show that these sets are relatively large.  
	
	In the next lemma we assume the following.
	Let $D'=(V,E)$ be a digraph on $n$ vertices, let $P = (u_1, \ldots, u_m)$ be a path of length $m \coloneqq m(n) \ge \frac{n}{4\log n}$ in $D'$ and let $E':=(V\times V)\setminus E$ be the set of all available edges in $D'$.
	
	\begin{lemma}
		\label{lem:ManyRotationPaths}
		Let $\frac{\log^{15}n}{n} \le p \in [0,1]$ and $q = \sqrt{\frac{p}{n \log^8 n}}$, and assume that $P$ contains no $\left(P, \frac{1}{8} \right)$-heavy vertices.
		Then with probability $1-o\left(\frac{1}{n \log n} \right)$ the following hold.
		There exist integers $t_{\ell} \coloneqq t_{\ell}(P)$ and $t_r \coloneqq t_r(P)$ satisfying $t_{\ell}, t_r \le \frac{\log n}{4\log\log n}$ and such that after performing online sprinkling double rotations to $P$ with $E'$ and $(t_{\ell},t_r)$, we have
		\begin{align}
		\label{eq:ENDq}
		\left|\END_{\ell}^{t_{\ell}} \right|, \left|\END_r^{t_r} \right| = \Omega\left(\tfrac{\log n}{\sqrt{q}} \right).
		\end{align}
	\end{lemma}
	
	The proof of \Cref{lem:ManyRotationPaths} is quite technical, thus we include it in \Cref{appendix:rotations}.

	\subsection{Merging two cycles}
	
	Here we describe our basic step of Phase $2$.
	That is, we show how to ``tailor'' two cycles in a given $1$-factor into one cycle consisting of the vertex set of their union.
	For doing so we will choose a \emph{designated} vertex from each cycle in each $1$-factor (more details about how we choose those designated vertices are given in \Cref{sec:desvxs}).
	Then, in the next section we show that it is possible to repeat the above ``tailoring'' procedure until each $1$-factor consists of a Hamilton cycle and they are all edge-disjoint.
	
	Let $C = (u_1, \ldots, u_b)$ and $C^* = (v_1, \ldots, v_a)$ be two cycles of lengths $b$ and $a$, respectively, and let $v_1 \in V(C^*)$ be a designated vertex of $C^*$.
	Let $E'$ be a set of available edges in $D'$ and $q \in [0,1]$ be an edge-exposure probability.
	In Procedure~\ref{procedure:2cycles} we describe our basic step, tailoring $C$ and $C^*$ together.
	\begin{algorithm}
		\caption{Merging $C$ and $C^*$}
		\label{procedure:2cycles}
		\textbf{INPUT}: $C = (u_1, \ldots, u_b), C^* = (v_1, \ldots, v_a), v_1, E', q$. \\
		\textbf{OUTPUT}: $C', E''$.
		\begin{algorithmic}[1]
			\STATE Let $E_1 \subset \left\{\overrightarrow{v_1u} \in E' ~:~ u_- \in V(C) \right\}$ be a subset of edges of size $|E_1| = \frac{b\log^7 n}{\sqrt{np}}$ chosen uniformly at random (if no such subset exists set $E_1 = \emptyset$).
			Expose all edges in $E_1$ independently at random with probability $q$.
			\STATE If neither of the edges in $E_1$ was successfully exposed output FAILURE.
			Otherwise, by re-labelling (if needed), assume that $\overrightarrow{v_1 u_1}$ was successfully exposed.
			\STATE Consider the path $P \coloneqq \left(v_2, \ldots, v_a, v_1, u_1, \ldots, u_b \right)$ which is obtained by deleting the edges $\overrightarrow{v_1v_2}, \overrightarrow{u_b u_1}$ and adding the edge $\overrightarrow{v_1u_1}$.
			\STATE Perform online sprinkling double rotations to $P$ with $E'$ and $H$ for $t_{\ell}$ and $t_r$ many times from left and right, respectively, where $t_{\ell}, t_r$ are minimal integers for which $\left|\END_{\ell}^{t_{\ell}} \right|, \left|\END_r^{t_r} \right| = \Omega\left(\frac{\log n}{\sqrt{q}} \right)$.
			If no such integers $t_{\ell}, t_r$ exist, output FAILURE.
			Otherwise, denote by $E^*$ the set successfully exposed edges.
			\STATE Expose all edges in the set $E_2 \coloneqq E'\left(\END_r^{t_r}, \END_{\ell}^{t_{\ell}} \right)$ independently with probability $q$.
			If no edge has been successfully exposed then output FAILURE.
			\STATE Let $\overrightarrow{yx} \in E_2$ be a successfully exposed edge and let $P'$ be a path with $x,y$ as left and right endpoints, respectively, as obtained by the above rotations of $P$.
			Let $C' = P' \cup \overrightarrow{yx}$ and $E'' = E'\setminus\left(E^* \cup \{\overrightarrow{v_1u_1}, \overrightarrow{yx}\} \right)$.
			\STATE Output $C'$ and $E''$.
		\end{algorithmic}
	\end{algorithm}
	
	\begin{lemma}
		\label{lem:join2cycles}
		Let $n$ be an integer and let $\frac{\log^{15}n}{n} \le p \le \varepsilon$ and $q = \sqrt{\frac{p}{n \log^8 n}}$.
		Let $C, C^*$ be two cycles of lengths $b$ and $a$, respectively, with $b\ge a$ and $b \ge \frac{n}{4\log n}$.
		Denote $V' \coloneqq V(C) \cup V(C^*)$.
		Let $E'$ be a set of available edges.
		And assume that the following hold.
		\begin{enumerate}
			\item $V'$ contains no $\left(V', \frac{1}{8} \right)$-heavy vertices.
			\item For any two subsets $U_1, U_2 \subset V'$ of sizes $|U_1|, |U_2|\geq  \frac{\log n}{100\sqrt{q}}$ we have $\left|E'(U_1, U_2) \right| \ge \frac{1}{2}|U_1||U_2|$.
		\end{enumerate}
		Then with probability $1 - o\left(\frac{1}{np\log n} \right)$, Procedure~\ref{procedure:2cycles} does not output FAILURE when applied to $C,C^*$ with $v_1, H, E', q$.
	\end{lemma}
	
	\begin{proof}
		We bound from above the probability to output a FAILURE in Procedure~\ref{procedure:2cycles} by bounding from above the probability to fail in each of the steps that can actually result in  FAILURE (that is, step $2$, step $4$ and step $5$).
		Denote $m \coloneqq a+b$.
		
		\paragraph*{FAILURE at step $2$:} First note that since $v_1$ is not $\left(V', \frac{1}{8} \right)$-heavy, we have
		\begin{align*}
		\left|\left\{\overrightarrow{v_1u} \in E' ~:~ u_- \in V(C) \right\} \right| &\ge \left|E'(v_1, V') \right| - \left|V(C^*) \right| \\
		&\ge \tfrac{7}{8}m - a\\
		&\ge \tfrac{b \log^7 n}{\sqrt{np}},
		\end{align*}
		so at step $1$ of the procedure we have $E_1 \neq \emptyset$.
		Now recall that step $2$ outputs FAILURE if and only if no edge in $E_1$ was successfully exposed at step $1$.
		Since $\frac{b \log^7 n}{\sqrt{np}}$ many edges were exposed, each with probability $q$, and since $b \ge \frac{n}{4\log n}$, we get that the probability of this event to occur is at most
		\begin{align*}
		(1-q)^{\frac{b\log^7 n}{\sqrt{np}}} \le e^{-\frac{qb \log^7 n}{\sqrt{np}}} \le e^{-\frac{q\sqrt{n}\log^6 n}{\sqrt{p}}} = e^{-\frac{1}{4}\log^2 n} = e^{-\omega(\log(np\log n))} = o\left(\tfrac{1}{np\log n} \right).
		\end{align*}	
		
		\paragraph*{FAILURE at step $4$:} The conditions of \Cref{lem:join2cycles} immediately imply that the conditions of \Cref{lem:ManyRotationPaths} are satisfied too.
		Hence, with probability $1 - o\left(\frac{1}{np\log n} \right)$ there exist integers $t_{\ell}, t_r$ as required, and we have $\left|\END_{\ell}^{t_{\ell}} \right|, \left|\END_r^{t_r} \right| = \Omega\left(\frac{\log n}{\sqrt{q}} \right)$.
		In particular, this means that with probability $1 - o\left(\frac{1}{np\log n} \right)$ Step $4$ does not output FAILURE.
		
		\paragraph*{FAILURE at step $5$:} As $\left|\END_{\ell}^{t_{\ell}} \right|, \left|\END_r^{t_r} \right| = \Omega\left(\frac{\log n}{\sqrt{q}} \right)$, by the assumption of the lemma we get that
		\[|E_2| = \left|E'\left(\END_r^{t_r}, \END_{\ell}^{t_{\ell}} \right) \right| \ge \tfrac{1}{2} \left|\END_{\ell}^{t_{\ell}} \right| \left|\END_r^{t_r} \right| =  \Omega\left(\tfrac{\log^2 n}{q} \right). \]
		Step $5$ outputs FAILURE if and only if no edge in $E_2$ was successfully exposed.
		But since each edge in $E_2$ was exposed independently with probability $q$, the probability of this event to occur is at most $(1-q)^{\Omega (\log^2 n / q)} \le e^{-\Omega (\log^2 n)} =n^{-\omega(1)}$.
	\end{proof}

	\subsection{\texorpdfstring{Turning a $1$-factor into a Hamilton cycle}{Turning a 1-factor into a Hamilton cycle}}
	
	In this section we describe how to convert a single $1$-factor into a Hamilton cycle by repeated applications of Procedure~\ref{procedure:2cycles}, merging its cycles one at a time, until we end up with a Hamilton cycle.
	Let $F=\left(C_0, \ldots, C_N \right)$ be a $1$-factor on $n$ vertices where $n$ and $N$ are integers and $C_0$ is of maximum length.
	For each $i \in [N]$, let $v_{C_i} \in V(C_i)$ be a designated vertex, and denote $\bar{v} = (v_{C_1}, \ldots, v_{C_N})$ and $V_i \coloneqq \bigcup_{j=0}^i V(C_j)$.
	Let $E'$ be a set of available edges and $q \in [0,1]$ be a probability parameter.
	In Procedure~\ref{procedure:factor} we convert $F$ into a Hamilton cycle.
	\begin{algorithm}
		\caption{Converting a $1$-factor into a Hamilton cycle}
		\label{procedure:factor}
		\textbf{INPUT}: $F = \left(C_0, C_1, \ldots, C_N \right), \bar{v} = (v_{C_1}, \ldots, v_{C_N}), E', q$. \\
		\textbf{OUTPUT}: $C, E''$.
		\begin{algorithmic}[1]
			\STATE Set $C \coloneqq C_0$ and $E'_1 \coloneqq E'$.
			\STATE \textbf{$i$th iteration:} For each $1\ \le i \le N$:
			\begin{ALC@g}
				\STATE Apply Procedure~\ref{procedure:2cycles} to $C, C_i$ with $v_{C_i}, E'_i, q$.
				\STATE If Procedure~\ref{procedure:2cycles} outputs FAILURE then output FAILURE as well.
				Otherwise, let $C'$ and $E'_{i+1}$ be the output of Procedure~\ref{procedure:2cycles} and update $C$ to be the new cycle $C'$.
			\end{ALC@g}
			\STATE Output $C$ and $E'' \coloneqq E'_{N+1}$.
		\end{algorithmic}
	\end{algorithm}
	
	\begin{lemma}
		\label{lem:factor}
		Let $n$ be an integer and let $\frac{\log^{15}n}{n} \le p \le \varepsilon$ and $q = \sqrt{\frac{p}{n \log^8 n}}$ be probability parameters.
		Let $F = (C_0, \ldots, C_N)$ be a $1$-factor on a vertex set $V$ of size $n$ and such that $C_0$ is of maximum length, and assume that $1\le N \le 4\log n$.
		Let $E'$ be a subset of available edges.
		For each $i \in [N]$ write $V_i \coloneqq \bigcup_{j=0}^i V(C_j)$,
		and assume that the following hold.
		\begin{enumerate}
			\item $V_i$ contains no $\left(V_i, \tfrac{1}{8} \right)$-heavy vertices.
			\item For any two subsets $U_1, U_2 \subset V$ of sizes $|U_1|, |U_2| = \Omega\left(\frac{\log n}{\sqrt{q}} \right)$ we have $\left|E'(U_1, U_2) \right| \ge \frac{1}{2}|U_1||U_2|$.
		\end{enumerate}
		Then with probability $1 - o\left(\frac{1}{np} \right)$, Procedure~\ref{procedure:factor} does not output FAILURE when applied to $F$ with $\bar{v}, E', q$.
	\end{lemma}
	
	\begin{proof}
		Procedure~\ref{procedure:factor} outputs FAILURE if and only if Procedure~\ref{procedure:2cycles} outputs FAILURE when applied at step $3$ for some $i \in [N]$.
		Hence, it is enough to show that whenever Procedure~\ref{procedure:2cycles} is applied, the conditions of \Cref{lem:join2cycles} are satisfied.
		
		Let $i \in [N]$.
		At the $i$th iteration of Procedure~\ref{procedure:factor}, we apply Procedure~\ref{procedure:2cycles} to the cycles $C, C_i$ as $C, C^*$, with $v_{C_i}, E'_i, q$.
		Note that conditions $1$--$2$ of this lemma imply conditions $1$--$2$ of \Cref{lem:join2cycles}, and moreover, we have $|C_0| \ge \frac{n}{N} \ge \frac{n}{4\log n}$.
		Hence, we get that with probability $1 - o\left(\frac{1}{np\log n} \right)$, Procedure~\ref{procedure:2cycles} does not output FAILURE when applied to $C, C_i$ with $v_{C_i}, E'_i$ and $q$.
		Taking a union bound over all $N \le 4\log n$ iterations of Procedure~\ref{procedure:factor}, we get that with probability $1 - o\left(\frac{1}{np} \right)$ it does not output FAILURE when applied to $F$ with $\bar{v}, E', q$.
	\end{proof}
	
	In the next section we will show that, with high probability, applying Procedure~\ref{procedure:factor} to all $1$-factors in the family $\mathcal F = \left\{F_1, \ldots, F_{\delta} \right\}$ obtained at the end of Phase 1, one by one, is possible without ever getting FAILURE as an output.
	
	
	\section{\texorpdfstring{Final procedure and proof of \Cref{thm:main3}}{Final Procedure and proof of Theorem 1.1}}
	
	In this section we describe the final procedure of Phase $2$ and we show that when applying it to the digraph $D'$ obtained at the end of Phase $1$ and not getting FAILURE as an output implies our main result.
	We start with presenting some properties of the digraph $D'$, or more generally of $D_{n,p}$, which will be useful in our final analysis in this section.
	Then we describe the final procedure and tie all the ingredients together and prove the main result.

	\subsection{Choosing designated vertices}
	\label{sec:desvxs}
	In order to turn $1$-factors in $\mathcal{F} = \left\{F_1, \ldots, F_{\delta} \right\}$ on the vertex set $V$ into Hamilton cycles, we start by choosing a set of \emph{designated} vertices.
	For each $k \in [\delta]$ and a $1$-factor $F_k = (C_{k,0}, \ldots, C_{k,N_k}) \in \mathcal F$ and for each $i \in [N_k]$ and a cycle $C_{k,i} \in F_k$, we choose a vertex $v_{k,i} \in V(C_{k,i})$ uniformly at random.
	Note that a vertex $v \in V$ can be chosen as a designated vertex at most once in each $1$-factor, but multiple times when going over all $1$-factors in $\mathcal F$.
	For $v \in V$, let $\Des_{\mathcal F}(v)$ denote that number of times $v$ is chosen as a designated vertex, or more formally,
	\begin{align*}
	\Des_{\mathcal F}(v) \coloneqq \left|\left\{i ~:~ \exists k \in [\delta] \text{ s.t. } v = v_{ki} \right\} \right|.
	\end{align*}
	In Procedure~\ref{procedure:2cycles} we expose edges adjacent to $v_{ki}$ in order to tailor $C_{ki}$ to another cycle.
	Since we cannot afford to expose edges adjacent to a certain vertex ``too many times'', it would be useful to have some upper bound on $\Des_{\mathcal F}(v)$ for all $v \in V$.
	This is the exact property that the following lemma captures. 
	
	\begin{lemma}
		\label{lem:keygnrl}
		Let $f(n)$ be a function satisfying $f(n) = \omega\left((np)^{1/3} \log n \right)$.
		Then with high probability we have $\Des_{\mathcal F}(v) \le f(n)$ for all $v \in V$.
	\end{lemma}
	
	For the proof of our main result, we only need to guarantee that with high probability each vertex is chosen at most $o\left(\sqrt{np} \right)$ many times, so we could choose $f(n)$ accordingly.
	Before proving \Cref{lem:keygnrl}, we need to introduce some useful notation.
	For a vertex $w \in V$ and an integer $i \in [\delta]$ we let $C_{w,i}$ denote the (unique) cycle in $F_i \in \mathcal F$ that contains $w$, and let $c_{w,i} \coloneqq |C_{w,i}|$ denote its length.
	The proof is based on the following lemma, which its proof, as being quite technical, is given in \Cref{sec:Appendix3Moment}.
	\begin{lemma}
		\label{lem:moment}
		For every $w \in [n]$ we have
		\begin{align*}
		\mathbb E \left[\left(\sum_{i \in [\delta]} \frac{1}{c_{w,i}} \right)^3 \right] = O\left(p \log^3 n \right).
		\end{align*}
	\end{lemma}
	
	Perhaps surprisingly, the proof of \Cref{lem:keygnrl} is obtained by using a third moment argument (as a second moment argument is not strong enough for our purposes).
	
	\begin{proof}[Proof of \Cref{lem:keygnrl}]
		By a third moment calculation and using \Cref{lem:moment}, we have
		\begin{align*}
		\Pr \left[\sum_{i \in [\delta]} \frac{1}{c_{w,i}} \ge f(n) \right] \le \frac{\mathbb E \left[\left(\sum_{i\in [\delta]}\frac{1}{c_{w,i}} \right)^3 \right]}{f(n)^3}
		= \frac{O\left(p\log^3 n \right)}{\omega \left(np \log^3 n \right)}
		= o\left(\frac{1}{n} \right).
		\end{align*}
		Taking a union bound over all $n$ vertices $w \in V$ finishes the proof.
	\end{proof}

	\subsection{The distribution of available edges}
	We need one more general auxiliary result about random graphs, in order to guarantee that enough available edges are available for us to close rotated paths into cycles.
	
	\begin{proposition}
		\label{prop:nonedges}
		Let $p \in \left(0, \frac{1}{5} \right)$.
		Let $n$ and $t \coloneqq t(n) = \omega(\log n)$.
		Then with high probability $D_{n,p}$ does not contain any two subsets $U_1, U_2 \subset [n]$ with $|U_1|, |U_2| = t$ such that $e(U_1, U_2) \ge \frac{1}{2}|U_1||U_2|$.
	\end{proposition}
	
	\begin{proof}
		Given two subsets $U_1, U_2 \subset [n]$ with $|U_1|, |U_2,| = t$, we have
		\[\Pr\left[e(U_1, U_2) \ge \tfrac{1}{2}t^2 \right] \le \binom{t^2}{t^2/2}p^{t^2/2} \le (2\sqrt{p})^{t^2} < \left(\tfrac{2}{\sqrt{5}} \right)^{t^2}. \]
		Taking a union bound over all pairs of subsets, we get that the probability of having a pair of subsets $(U_1, U_2)$ with $|U_1|=|U_2|=t = \omega(\log n)$ and $e(U_1, U_2) \ge \frac{1}{2}|U_1||U_2|$ is at most
		\[\binom{n}{t}^2 \left(\tfrac{2}{\sqrt{5}} \right)^{t^2} \le n^{2t} \left(\tfrac{2}{\sqrt{5}} \right)^{t^2} = o(1). \qedhere \]
	\end{proof}

	\subsection{Final procedure}
	
	Let $\mathcal F = \left\{F_1, \ldots, F_{\delta} \right\}$ be a family of edge-disjoint $1$-factors on a vertex set $V$ of size $n$.
	For each $k \in [\delta]$ write $F_k = (C_{k,0}, \ldots, C_{k,N_k}) \in \mathcal F$, where $N_k$ is an integer, and let $\bar{v}_k\coloneqq \left(v_{k,1}, \ldots, v_{k,N_k} \right)$ be a tuple of designated vertices 
	such that $v_{k,i} \in V(C_{k,i})$ for each $i \in [N_k]$.
	Procedure~\ref{procedure:final} is the final procedure of our proof, applied to a family of $\delta$ edge-disjoint $1$-factors $\mathcal F$, with $\left(\bar{v}_1, \ldots, \bar{v}_{\delta} \right)$ as above, a set of available edges $E'$, and an edge-exposure probability $q \in [0,1]$.
	\begin{algorithm}
		\caption{Converting $\mathcal F$ into a family $\mathcal H_{\delta}$ of Hamilton cycles}
		\label{procedure:final}
		\textbf{INPUT}: $\mathcal F = \left\{F_1, \ldots, F_{\delta} \right\}, \left(\bar{v}_1, \ldots, \bar{v}_{\delta} \right), E', q$. \\
		\textbf{OUTPUT}: $\mathcal H_{\delta}$.
		\begin{algorithmic}[1]
			\STATE Set $\mathcal H_0 \coloneqq \emptyset$ and $E'_1 = E'$.
			\STATE \textbf{$k$th iteration}: For each $1\le k \le \delta$:
			\begin{ALC@g}
				\STATE Apply Procedure~\ref{procedure:factor} to $F_k$ with $\bar{v}_k, E'_k, q$.
				\STATE If Procedure~\ref{procedure:factor} outputs FAILURE, end procedure and output FAILURE.
				Otherwise, let $C_k$ be the output Hamilton cycle and $E'_{k+1}$ be the output set of available edges from Procedure~\ref{procedure:factor}.
				Let $\mathcal H_k = \mathcal H_{k-1} \cup \{C_k \}$.
			\end{ALC@g}
			\STATE Output $\mathcal H_{\delta}$.
		\end{algorithmic}
	\end{algorithm}
	
	Consider the digraph $D'$ obtained at the end of Phase $1$, and let $E'$ be the set of available edges in it.
	Let $\mathcal F = \left\{F_1, \ldots, F_{\delta} \right\}$ be the family of $1$-factors where $\delta = \delta^{\pm}(D_{n,p}) = (1+o(1))np$, and such that for each $k \in [\delta]$ we write $F_k = \left(C_{k,0}, \ldots, C_{k,N_k} \right)$ where $C_{0,k}$ is of maximum length among all cycles in $F_k$.
	Recall that $p_1 = \sqrt{\frac{p}{n\log^4 n}}$ and $q = \sqrt{\frac{p}{n\log^8 n}} = \frac{p_1}{\log^2 n}$.
	For $k \in [\delta]$ and $i \in [N_k]$ define the following.
	\begin{enumerate}
		\item $v_{k,i} \in V(C_{k,i})$ the designated vertex chosen from the cycle $C_{k,i}$ for the $1$-factor $F_k$.
		\item $\bar{v}_k \coloneqq \left(v_{k,1}, \ldots, v_{k,N_k} \right)$.
		\item $V_{k,i} \coloneqq \bigcup_{j=0}^{i} V(C_{k,j})$.
	\end{enumerate}

	\begin{lemma}
		\label{lem:final}
		Let $D'$ be the digraph obtained at the end of Phase 1, let $E'$ be the set of available edges in it and let $\mathcal F = \left\{F_1, \ldots, F_{\delta} \right\}$ be the family of $1$-factors.
		Then using the notation above, with high probability Procedure~\ref{procedure:final} does not output FAILURE when applied to $\mathcal F$ with $\left(\bar{v}_1, \ldots, \bar{v}_{\delta} \right), E', q$.
	\end{lemma}
	
	\begin{proof}
		Procedure~\ref{procedure:final} outputs FAILURE if and only if at step $2$ for some iteration $k \in [\delta]$ Procedure~\ref{procedure:factor} outputs FAILURE when applied to $F_k, \bar{v_k}, E'_k, q$.
		Hence, it is enough to show that whenever Procedure~\ref{procedure:factor} is applied, the conditions of \Cref{lem:factor} are satisfied.
		
		Firstly, by \Cref{lem:fewcycles} and \Cref{lem:longcycle}, with high probability we have $N_k \le 4\log n$ and $\big|C_{k,0} \big| \ge \frac{n}{4\log n}$ for all $k \in [\delta]$.
		By \Cref{prop:nonedges} we further get that, with high probability, every two subsets $U_1, U_2 \subset V(D')$ of size $|U_1|, |U_2| = \Omega\left(\frac{\log n}{\sqrt{q}} \right)$ satisfy $\left|E'(U_1, U_2) \right| \ge \frac{1}{2}|U_1||U_2|$.
		
		Next, we show that with high probability for all $k \in [\delta]$ condition $1$ of \Cref{lem:factor} holds for all $i \in [N_k]$.
		For all $k \in [\delta]$ and all $i \in [N_k]$, by the definition of a $\left(V_{k,i}, \frac{1}{9} \right)$-heavy vertex and by \Cref{prop:heavyvtx}, we know that with high probability every $u \in V_{k,i}$ satisfy
		\begin{align}
		\label{eq:E11u}
		\left|E'\left(V_{k,i}, u \right) \right| \ge \tfrac{8}{9}m_{k,i} && \text{ and } && \left|E'\left(u, V_{k,i} \right) \right| \ge \tfrac{8}{9}m_{k,i}.
		\end{align}
		We show that condition $1$ from \Cref{lem:factor} holds for $E'_k$.
		In fact, we show the slightly stronger following statement.
		For $k \in [\delta]$ and $i \in [N_k]$ consider the $(k,i)$th iteration of Procedure~\ref{procedure:final}, that is, the $i$th iteration of Procedure~\ref{procedure:factor} in the $k$th iteration of Procedure~\ref{procedure:final}, where we apply Procedure~\ref{procedure:2cycles} to $C, C_{k,i}$ with $v_{k,i}, E'_{k,i}$ and $q$, where $E'_{i,k}$ is the set of available edge in the input.
		We show that, with high probability, for all $(k,i)$ and for all $u \in V_{k,i}$ we have
		\begin{align}
		\label{eq:inEkiu}
		\left|E'_{k,i}\left(V_{k,i}, u \right) \right| \ge \tfrac{7}{8}m_{k,i}, 
		\end{align}
		and
		\begin{align}
		\label{eq:outEkiu}
		\left|E'_{k,i}\left(u, V_{k,i} \right) \right| \ge \tfrac{7}{8}m_{k,i}.
		\end{align}
		
		Given a vertex $u \in V(D')$, the above quantities can decrease in an iteration of Procedure~\ref{procedure:2cycles} in one of the following cases.
		\begin{enumerate}
			\item[(a)] At step $2$ if $u = v_1$.
			\item[(b)] At step $2$ if $u = u_1$.
			\item[(c)] At step $4$ if an edge going out of $u$ is in $E^*$.
			\item[(d)] At step $4$ if an edge going into $u$ is in $E^*$.
			\item[(e)] At step $6$ if $u=y$.
			\item[(f)] At step $6$ if $u=x$.
		\end{enumerate}
		Denote by $Y_a(u), Y_b(u), Y_c(u), Y_d(u), Y_e(u), Y_f(u)$ the random variables counting the numbers of times that cases (a) -- (f) occured, respectively.
		
		\paragraph*{Proof of (\ref{eq:inEkiu}).}
		Let $k \in [\delta]$ and $i \in [N_k]$, and consider the $(k,i)$th iteration of Procedure~\ref{procedure:2cycles}, where $E'_{k,i}$ is the set of available edges in the input.
		If $(k,i) = (1,1)$ then we have $E'_{1,1} = E'$ and then (\ref{eq:inEkiu}) follows immediately by (\ref{eq:E11u}).
		If $(k,i) \neq (1,1)$ then by Procedure~\ref{procedure:2cycles} we have
		\begin{align}
		\label{eq:inEk1-Y}
		\left|E'_{k,i}\left(V_{k,i}, u \right) \right| \ge \left|E'_{1,1}\left(V_{k,i}, u \right) \right| - \left(Y_a(u) + Y_c(u) + Y_e(u) \right).
		\end{align}
		
		In a given iteration $(k,i)$, case (a) occurs at most once, and in particular it means that the edge $\overrightarrow{v_{k,i} u}$ was successfully expose.
		The probability for this event is at most
		\[q\cdot \tfrac{m_{k,i} \log^7 n}{\sqrt{np}} \cdot \tfrac{1}{m_{k,i}} \le \tfrac{\log^3 n}{n}. \]
		By \Cref{lem:fewcycles} and since $\delta = O(np)$, with high probability there are at most $O(np\log n)$ many iteration, and hence $Y_a(u)$ is stochastically dominated by a binomial random variable $Y_a \sim \Bin\left(O(np\log n), \frac{\log^3 n}{n} \right)$.
		by \Cref{Chernoff}, with $\eta = \frac{\log^2 n}{p} -1$, we get
		\begin{align}
		\begin{split}
		\label{eq:casea}
		\Pr\left[Y_a(u) \ge \log^7 n \right] &\le \Pr\left[Y_a \ge \log^7 n \right] \\
		&\le e^{-\eta^2\cdot O(p\log^5 n)/(2+\eta)} \\
		&\le e^{-\log^6 n} \\
		&= o\left(\tfrac{1}{n^2p\log^2 n} \right).
		\end{split}
		\end{align}
		
		Similarly, in a given iteration $(k,i)$, case (e) occurs at most once, and in particular it means that $u$ was chosen as a right enpoint in $\END^{t_r}_r$, for an appropriate $t_r$, which, by \Cref{lem:ManyRotationPaths}, happens with probability at most
		\[\Theta\left(\tfrac{\log n}{m_{k,i}\sqrt{q}} \right) = O\left(\tfrac{\log^2 n}{n\sqrt{q}} \right).\]
		In this case, we know that all edges from $\END_{\ell}^{t_{\ell}}$ into $u$ were exposed with probability $q$ to be successfully exposed.
		By \Cref{lem:ManyRotationPaths}, the probability of case (e) to occur on iteration $(k,i)$ is at most
		\[O\left(\tfrac{\log^2 n}{n\sqrt{q}} \right) \cdot \tfrac{\log n}{\sqrt{q}} \cdot q = O\left(\tfrac{\log^3 n}{n} \right). \]
		Considering all iterations, we get that with high probability $Y_e(u)$ is stochastically dominated by a binomial random variable $Y_e \sim \Bin\left(O(np\log n), O\left(\frac{\log^3 n}{n} \right) \right)$, and by a similar calculation to the one in the analysis above of case (a), we get that
		\begin{align}
		\label{eq:casee}
		\Pr\left[Y_e(u) \ge \log^7 n \right] = o\left(\tfrac{1}{n^2p\log n} \right).
		\end{align}
		
		As for case (c), in a given iteration $(k,i)$ it can occur more than once.
		This happens when $u$ is a pivot for online sprinkling double rotation, which by the proof of \Cref{lem:ManyRotationPaths} happens with probability at most
		\[\Theta\left(\tfrac{\log n}{m_{k,i}\sqrt{q}} \right) = O\left(\tfrac{\log^2 n}{n\sqrt{q}} \right). \]
		Then, by the proof of \Cref{lem:ManyRotationPaths}, we know that with high probability at most $\log^3 n$ edges are exposed, each independently with probability $q$.
		In total, the probability that edge is counted in $Y_c(u)$ is at most $\frac{\sqrt{q} \log^2 n}{n} = \frac{p^{1/4}}{n^{5/4}}$, and over all iterations this happens at most $O(np\log^5 n)$ many times.
		We get that $Y_c(u)$ is stochastically dominated by a binomial random variable $Y_c\sim \Bin\left(O(np\log^5 n), \frac{p^{1/4}}{n^{5/4}} \right)$.
		Again, by \Cref{Chernoff}, with $\eta = \frac{n^{1/4}\log^3 n}{p^{5/4}} - 1$, we get
		\begin{align}
		\label{eq:casec}
		\Pr\left[Y_e(u) \ge \log^7 n \right] \le \Pr\left[Y_e \ge \log^7 n \right] \le e^{-\log^6 n} = o\left(\tfrac{1}{n^2p\log n} \right).
		\end{align}
		
		By (\ref{eq:E11u}), (\ref{eq:inEk1-Y}), (\ref{eq:casea}), (\ref{eq:casee}), (\ref{eq:casec}) and since $m_{k,i} \ge \frac{n}{4\log n}$, and by taking a union bound over all iterations and vertices, we get that with high probability we have
		\[\left|E'_{k,i}\left(V_{k,i}, u \right) \right| \ge \tfrac{8}{9}m_{k,i} - 3\log^7 n \ge \tfrac{7}{8}m_{k,i}, \]
		for all $k \in [\delta]$, $i \in [N_k]$ and $u \in V_{k,i}$, as required.
		
		\paragraph*{Proof of (\ref{eq:outEkiu}).}
		Let $k \in [\delta]$ and $i \in [N_k]$, and consider the $(k,i)$th iteration of Procedure~\ref{procedure:2cycles}, where $E'_{k,i}$ is the set of available edges in the input.
		If $(k,i) = (1,1)$ then we have $E'_{1,1} = E'$ and then (\ref{eq:inEkiu}) follows immediately by (\ref{eq:E11u}).
		If $(k,i) \neq (1,1)$ then by Procedure~\ref{procedure:2cycles} we have
		\begin{align}
		\label{eq:outEk1-Y}
		\left|E'_{k,i}\left(u, V_{k,i} \right) \right| \ge \left|E'_{1,1}\left(u, V_{k,i} \right) \right| - \left(Y_b(u) + Y_d(u) + Y_f(u) \right).
		\end{align}
		
		The rest of the proof is almost similar to the proof of (\ref{eq:inEkiu}), where the analysis of case (d) is similar to the analysis of case (c), and the analysis of case (f) is similar to the analysis of case (e).
		However, the analysis of case (b) is different to the analysis of case (a) so we include it here and omit the other two cases.
		
		In a given iteration $(k,i)$, case (b) occurs at most once, and in particular it means that $u = v_{k,i}$.
		Hence, $Y_b(u) \le \Des_{\mathcal F}(u)$, and by \Cref{lem:keygnrl}, we know that with high probability we have $Y_b(u) \le (np)^{1/3}\log^2 n$.
		
		In total we get that, with high probability, for all $k \in [\delta]$, $i \in [N_k]$ and $u \in V_{k,i}$ we have
		\[\left|E'_{k,i}\left(u, V_{k,i} \right) \right| \ge \tfrac{8}{9}m_{k,i} - 2\log^7 n - (np)^{1/3}\log^2 n \ge \tfrac{7}{8}m_{k,i}, \]
		as required.
		
		Hence, by \Cref{lem:factor} we get that with probability $1 - o\left(\frac{1}{np} \right)$ Procedure~\ref{procedure:factor} does not output FAILURE when applied on the $k$th iteration in step $2$ of Procedure~\ref{procedure:final}.
		Taking a union bound over all $\delta = (1+o(1))np$ iterations we get that with high probability Procedure~\ref{procedure:final} does not output FAILURE when applied to $\mathcal F$ with $(\bar{v_1}, \ldots, \bar{v_{\delta}}), E', q$ as in the statement.
	\end{proof}

	\subsection{Proof of main result}
	To prove \Cref{thm:main3} we show the following.
	\begin{enumerate}
		\item One can apply Phase $2$ to the digraph $D'$ obtained at the end of Phase $1$ with appropriate parameters.
		\item The outcome digraph $D''$ is equivalent to generating $D_{n,p}$.
		\item The Hamilton cycles in the family $\mathcal H_{\delta}$ in the outcome digraph are edge-disjoint.
	\end{enumerate}
	These three ingredients together give a decomposition of $D_{n,p}$ into minimum-degree many Hamilton cycles.
	\Cref{lem:final} implies the first item.
	The third item is implied by the fact that in Procedure~\ref{procedure:2cycles} each edge which is used in a newly formed Hamilton cycle is removed from the set of available edge for future iterations.
	As for the second item, we show that during the second phase, no edge was exposed ``too many'' times.
	
	Let $D'$ be the digraph obtained at the end of Phase $1$, let $E'$ be the set of available edges in it and let $\mathcal F = \left\{F_1, \ldots, F_{\delta} \right\}$ be the family of $1$-factors.
	Recall that $D'$ is a digraph on $n$ vertices and $\delta = (1+o(1))np$ where $\frac{\log^{15}n}{n} \le p < \varepsilon$.
	Suppose that, using the notation above, Procedure~\ref{procedure:final} is applied to $\mathcal F$ with $\left(\bar{v}_1, \ldots, \bar{v}_{\delta} \right), E', q$, where $q = \frac{p_1}{\log^2 n}$ and $p_1 = \sqrt{\frac{p}{n\log^4 n}}$.
	Given a directed edge $e \in E'$ let $X_e$ be the random variable counting the number of times the edge $e$ was exposed during Procedure~\ref{procedure:final}.
	
	\begin{lemma}
		\label{lem:uvProb}
		With high probability for all $e \in E'$ we have $X_e = o(\log^2 n)$.
	\end{lemma}
	
	Given \Cref{lem:uvProb}, with high probability each edge in Phase $2$ was exposed $o(\log^2 n)$ many times with probability $q = \frac{p_1}{\log^2 n}$.
	Hence, with high probability, each edge was exposed with probability $o(p_1)$, implying that the digraph generated in Phase $2$ is a directed subgraph of $D_{n,p_1}$, and in total $D''$ is a directed subgraph of $D_{n,p}$, finishing the proof.
	
	It is left to prove \Cref{lem:uvProb}.
	\begin{proof}[Proof of \Cref{lem:uvProb}]
		For a vertex $v \in V(D')$, we consider several quantities related to $v$ during Phase $2$.
		Denote by $\e_r(v)$ and $\e_{\ell}(v)$ the numbers of times $v$ was chosen as a right and left endpoint of a path when performing online sprinkling right and left rotations, respectively.
		Denote by $\m_r(v)$ and $\m_{\ell}(v)$ the numbers of times $v$ was chosen as a pitot vertex which is not an endpoint for right and left rotations of a path, respectively.
		Recall that by $\Des_{\mathcal F}(v)$ we denote the number of times where $v$ was chosen as a designated vertex in some cycle of some $1$-factor in $\mathcal F$.
		The following claim gives an estimation of these quantities.
		\begin{claim}
			\label{clm:uvProb4}
			With high probablity for all $v \in V(D')$ we have
			\[\e_r(v), \e_{\ell}(v), \m_r(v), \m_{\ell}(v) \le n^{1/4}p^{3/4}\log^7 n. \]
		\end{claim}
		
		\begin{proof}
			We prove the statement for $\e_{\ell}(u)$ and $\m_{\ell}(u)$, where the proofs for $\e_r(u)$ and $\m_r(u)$ are similar.
			By the proof of \Cref{lem:ManyRotationPaths}, we get that for each $k \in [\delta]$ and $i \in [N_k]$, both probabilities of $u$ to be an endpoint and a mid-point in a left rotation of the path on $V_{ik}$, are stochastically dominated by the uniform distribution of choosing a subset of vertices of size $\Theta\left(\frac{\log n}{\sqrt{q}} \right) = \Theta\left(\frac{n^{1/4}\log^3 n}{p^{1/4}} \right)$ out of at least $\frac{n}{8\log n}$ vertices.
			Recall that with high probability we have $\delta = O(np)$ and $N_k \le 4\log n < \log^2 n$ for all $k \in [\delta]$, then by considering all cycles in all $1$-factors we get that both $\e_{\ell}(u)$ and $\m_{\ell}(u)$ are stochastically dominated by the random variable $X \sim \Bin\left(np\log^2 n, \frac{\alpha\log^4 n}{n^{3/4}p^{1/4}} \right)$ for some positive constant $\alpha$.
			Hence, by \Cref{Chernoff}, taking $\eta = \frac{n^{1/12}}{2p^{5/12}\log^{11}n} - 1$, we get
			\begin{align*}
			\Pr\left[\e_{\ell}(u) \ge n^{1/4}p^{3/4}\log^7 n \right] \le \Pr\left[X \ge n^{1/4}p^{3/4}\log^7 n \right] 
			\le e^{-\frac{\eta^2 \mathbb E[X]}{2+\eta}} 
			< e^{-\frac{(np)^{1/3}}{2\log^4 n}}
			= o\left(\tfrac{1}{n} \right),
			\end{align*}
			and similarly
			\begin{align*}
			\Pr\left[\m_{\ell}(u) \ge n^{1/4}p^{3/4}\log^7 n \right] = o\left(\tfrac{1}{n} \right).
			\end{align*}
			The statement then follows by taking a union bound over all $u \in V(D')$.
		\end{proof}

		Let $e = \overrightarrow{uv} \in E'$ be a directed edge.
		We wish to bound the expected number of times the edge $\overrightarrow{uv}$ was exposed during Phase $2$, that is, we show that
		\begin{align*}
		\Pr\left[\exists e\in E' \text{ s.t. } X_e = \Omega(\log^2 n) \right] = o(1).
		\end{align*}
		
		Let $k \in [\delta]$ and consider the $k$th iteration of step $2$ of Procedure~\ref{procedure:final}, where Procedure~\ref{procedure:factor} is applied to $F_k$ with $\bar{v}_k, E'$ and $q$.
		Let $F_k = (C_{k,1}, \ldots, C_{k,N_k})$ be the cycles in this $1$-factor, where $N_k \le 4\log n < \log^2 n$ is an integer.
		Let $i \in [N_k]$ and consider the $i$th iteration of Procedure~\ref{procedure:factor}, where Procedure~\ref{procedure:2cycles} is applied to merge two cycles.
		The input of Procedure~\ref{procedure:2cycles} at this point is $C, C_{k,i}, v_{k,i}, E'$ and $q$, where $C$ is a cycle with $V(C) = \bigcup_{j=0}^{i-1}V(C_{k,j})$, obtained as an output of the previous iteration of Procedure~\ref{procedure:factor}.
		The edge $\overrightarrow{uv}$ was exposed in Procedure~\ref{procedure:2cycles} if one of the following cases occured.
		\begin{itemize}
			\item In step $1$ if $\overrightarrow{uv} \in E_1$.
			\item In step $4$ when performing online sprinkling double rotations to $P$ with $E'$.
			\item In step $5$ if $\overrightarrow{uv} \in E_2$.
		\end{itemize}
		We analyze each case separately.
		
		Let $X^1_e, X^4_e$ and $X^5_e$ be the random variables that count the number of times the directed edge $e$ was exposed in steps $1$, $4$ and $5$ of Procedure~\ref{procedure:2cycles}, respectively.
		We have
		\[X_e = X^1_e + X^4_e + X^5_e\]
		and so
		\begin{align}
		\begin{split}
		\label{eq:uvProb145}
		\Pr\left[\exists e \in E' \text{ s.t. } X_e = \Omega(\log^2 n)  \right] &\le \Pr\left[\exists e \in E' \text{ s.t. } X^1_e = \Omega(\log^2 n)  \right] \\
		&+ \Pr\left[\exists e \in E' \text{ s.t. } X^4_e = \Omega(\log^2 n)  \right] \\
		&+ \Pr\left[\exists e \in E' \text{ s.t. } X^5_e = \Omega(\log^2 n) \right].
		\end{split}
		\end{align}
		We bound each term separately.
		We will use the following relatively weak tail bound for a binomial random variable.
		Given $X \sim \Bin(n,p)$ such that $np < 1$ and an integer $k$, we have
		\begin{align*}
		\Pr\left[X \ge k \right] = \sum_{i=k}^n \binom{n}{i}p^i(1-p)^i \le \sum_{i=k}^n (np)^i < n(np)^k.
		\end{align*}
		
		\paragraph*{First term in (\ref{eq:uvProb145}).}
		Note that if the edge $e = \overrightarrow{uv}$ was exposed in step $1$ of Procedure~\ref{procedure:2cycles}, on iteration $(k,i)$ of Procedure~\ref{procedure:final}, then in particular we have $u = v_{k,i}$ and $v$.
		Given that $u = v_{k,i}$ each edge from $u$ to successors of vertices in $V(C)$ is chosen uniformly at random with probability $\frac{\log^7 n}{\sqrt{np}}$.
		Hence, $X^1_e$ is stochastically dominated by a binomial random variable $\Bin\left(\Des_{\mathcal F}(u), \frac{\log^7 n}{\sqrt{np}} \right)$.
		By \Cref{lem:keygnrl}, we know that with high probability we have that $\Des_{\mathcal F}(u) \le (np)^{1/3}\log^2 n$ holds for all $u \in V(D')$.
		Hence, with high probability, for all $e \in E'$ we get that $X^1_e$ is stochastically dominated by the binomial random variable $Y^1 \sim \Bin\left((np)^{1/3}\log^2 n, \frac{\log^7 n}{\sqrt{np}} \right)$.
		In particular,
		\begin{align*}
		\Pr\left[X^1_e = \Omega(\log^2 n) \right] &\le \Pr\left[Y^1 = \Omega(\log^2 n) \right] \le (np)^{1/3}\log^2 n \left(\tfrac{\log^7 n}{\sqrt{np}} \right)^{\Omega(\log^2 n)} = o\left(\tfrac{1}{n^2} \right).
		\end{align*}
		Hence, by taking a union bound over all $e \in E'$ and by the above,
		\begin{align*}
		\Pr\left[\exists e\in E' \text{ s.t. } X^1_e = \Omega(\log^2 n) \right] \le n^2 \Pr\left[X^1_e = \Omega(\log^2 n) \right] = o(1).
		\end{align*}

		\paragraph*{Second term in (\ref{eq:uvProb145}).}
		If the edge $e=\overrightarrow{uv}$ was exposed in step $4$ of Procedure~\ref{procedure:2cycles} on iteration $(k,i)$ of Procedure~\ref{procedure:final}, then in particular one of the following occured.
		Either it was exposed during online sprinkling right rotations with
		\begin{itemize}
			\item $u$ as a right endpoint of a path, or
			\item $u_+$ as a pivot vertex in $V_3$ of a path,
		\end{itemize}
		or it was exposed during online sprinkling left rotations of a path $P$ with
		\begin{itemize}
			\item $v$ as a left endpoint of a path, or
			\item $v_-$ as a pivot vertex in $V_2$ of a path.
		\end{itemize}
		By \Cref{clm:uvProb4}, with high probability for all $u, v \in V(D')$ we have $\e_r(u) + \m_r(u_+) + \e_{\ell}(v) + \m_{\ell}(v_-) \le 4n^{1/4}p^{3/4}\log^7 n$.
		Hence, with high probability $X^4_e$ is stochastically dominated by the binomial random variable $Y^4 \sim \Bin\left(4n^{1/4}p^{3/4}\log^7 n, \frac{100\log n}{m_{k,i}} \right)$.
		Recall $m_{k,i} \ge \frac{n}{4\log n}$ and $q = \sqrt{\frac{p}{n\log^8 n}}$, so we get that
		\begin{align*}
		\Pr\left[X^4_e = \Omega(\log^2 n) \right] &\le \Pr\left[Y^4 = \Omega(\log^2 n) \right] \\
		&\le 4n^{1/4}p^{3/4}\log^7 n \cdot \left(\tfrac{400n^{1/4}p^{3/4}\log^8 n}{m_{k,i}} \right)^{\Omega(\log^2 n)} \\
		&\le n^{1/2} \left(\tfrac{p^{3/4}\log^9 n}{n^{3/4}} \right)^{\Omega(\log^2 n)} \\
		&= o\left(\tfrac{1}{n^2} \right).
		\end{align*}
		By taking a union bound over all $e \in E'$ and by the above we get
		\begin{align*}
		\Pr\left[\exists e\in E' \text{ s.t. } X^4_e = \Omega(\log^2 n) \right] = o(1).
		\end{align*}

		\paragraph*{Third term in (\ref{eq:uvProb145}).}
		If the edge $e = \overrightarrow{uv}$ was exposed in step $5$ of Procedure~\ref{procedure:2cycles} on iteration $(k,i)$ of Procedure~\ref{procedure:final}, then in particular we have $u \in \END_r^{t_r}$ and $v\in \END_{\ell}^{t_{\ell}}$, for the right choice of parameters $t_r, t_{\ell}$.
		Recall, similarly to the proof of \Cref{lem:ManyRotationPaths}, that the probablity of $v$ to be a left endpoint when performing online sprinkling left rotations to a path is stochastically dominated by the uniform distribution of choosing a subset of vertices of size $\Theta\left(\frac{\log n}{\sqrt{q}} \right) = \Theta\left(\frac{n^{1/4}\log^3 n}{p^{1/4}} \right)$ out of $\frac{1}{2}m_{k,i} \ge \frac{n}{8\log n}$ many vertices.
		On the other hand, by \Cref{clm:uvProb4} we know that with high probability we have $\e_r(u) \le n^{1/4}p^{3/4}\log^7 n$.
		Hence, $X^5_e$ is stochastically dominated by the binomial random variable $Y^5 \sim \Bin\left(n^{1/4}p^{3/4}\log^7 n, \frac{\alpha \log^4 n}{n^{3/4}p^{1/4}} \right)$, for some constant $\alpha$.
		Thus we get
		\begin{align*}
		\Pr\left[X^5_e = \Omega(\log^2 n) \right] \le  \Pr\left[Y^5 = \Omega(\log^2 n) \right]
		\le n^{1/4}p^{3/4}\log^7 n \left(\tfrac{\alpha p^{1/2}\log^{11}n}{n^{1/2}} \right)^{\Omega(\log^2 n)} =o\left(\tfrac{1}{n^2} \right).
		\end{align*}
		Taking a union bound over all $e \in E'$ and by the above we get
		\begin{align*}
		\Pr\left[\exists e \in E' \text{ s.t. } X^5_e = \Omega(\log^2 n) \right] = o(1).
		\end{align*}
		\\
		
		In total, we get that
		\begin{align*}
		\Pr\left[\exists e \in E' \text{ s.t. } X_e = \Omega(\log^2 n) \right] = o(1),
		\end{align*}
		as required.
	\end{proof}
\vspace{20pt}
{\bf Acknowledgement.} The first author would like to thank Kaarel Haenni for many fruitful discussions.
The second author was supported by UK Research and Innovation grant MR/W007320/2.
	
\bibliographystyle{abbrv}
\bibliography{refs}

\appendix

\section{\texorpdfstring{Proof of \Cref{lem:ManyRotationPaths}}{Proof of Lemma 4.2}}
\label{appendix:rotations}
We start with reminding the lemma proved in this appendix and the relevant notation.
Let $D'=(V,E)$ be a digraph on $n$ vertices, let $P = (u_1, \ldots, u_m)$ be a path of length $m \coloneqq m(n) \ge \frac{n}{4\log n}$ in $D'$ and let $E':=(V\times V)\setminus E$ be the set of all available edges in $D'$.

\begin{replemma}{lem:ManyRotationPaths}
	Let $\frac{\log^{15}n}{n} \le p \in [0,1]$ and $q = \sqrt{\frac{p}{n \log^8 n}}$, and assume that $P$ contains no $\left(P, \frac{1}{8} \right)$-heavy vertices.
	Then with probability $1-o\left(\frac{1}{n \log n} \right)$ the following hold.
	There exist integers $t_{\ell} \coloneqq t_{\ell}(P)$ and $t_r \coloneqq t_r(P)$ satisfying $t_{\ell}, t_r \le \frac{\log n}{4\log\log n}$ and such that after performing online sprinkling double rotations to $P$ with $E'$ and $(t_{\ell},t_r)$, we have
	\begin{align}
	\left|\END_{\ell}^{t_{\ell}} \right|, \left|\END_r^{t_r} \right| = \Omega\left(\tfrac{\log n}{\sqrt{q}} \right). \tag{\ref{eq:ENDq}}
	\end{align}
\end{replemma}

The proof of \Cref{lem:ManyRotationPaths} is quite technical, thus we include it in Appendix.

\begin{proof}
	We prove the statement for $\END_{\ell}^{t_{\ell}}$, and the proof for $\END_r^{t_r}$ is similar.
	In fact, we prove the following stronger statement and show that it implies (\ref{eq:ENDq}).
	\begin{claim}
		\label{cl:ENDqell}
		Let $t \coloneqq t(n)$ be an integer satisfying
		\begin{align}
		\label{eq:qp}
		\left(2\log n \right)^{2t} = O\left(\tfrac{\log n}{\sqrt{q}} \right).
		\end{align}
		Then with probability $1 - o\left(\frac{1}{n \log n} \right)$ we have
		\begin{align}
		\label{eq:ENDqell}
		\left(2\log n \right)^{2t} \le \left|\END_{\ell}^t \right| \le \left(50\log n \right)^{2t}.
		\end{align}
	\end{claim}
	
	\begin{proof}[Proof of \Cref{cl:ENDqell}]
		We proceed by induction on $t$.
		If $t=0$ then (\ref{eq:ENDqell}) holds trivially, as $\END_{\ell}^0 = \{u_1 \}$.
		Let $t \ge 1$ and assume that the statement holds for all $0 \le t' \le t-1$.
		
		Observe that, using the notation introduced for online sprinkling left rotations of a path, one can write
		\begin{align}
		\label{eq:ENDunion}
		\END_{\ell}^t = \bigcup_{u \in \END_{\ell}^{t-1}} \END_{\ell}^1(P^u) = \bigcup_{u \in \END_{\ell}^{t-1}} \bigcup_{y \in \IN{u}{V_2(P^u)}} \left(\IN{y_+}{V_1(P^u)} \right)^+.
		\end{align}
		We show that, with high probability, the union in (\ref{eq:ENDunion}) consists of sets which are ``almost'' pairwise disjoint.
		That is, we show that, with high probability, every vertex in $V(P)$ is the endpoint of at most two different paths obtained from $P$ after performing $t-1$ many rotations.
		Assuming that, it follows that the size of the union in (\ref{eq:ENDunion}) is at least half of the sum of the sizes of the sets.
		
		Recall that for any $u \in \END_{\ell}^{t-1}$ we have $V_1(P^u)\cup V_2(P^u) = V_1(P)\cup V_2(P)$, since performing left rotations to a path does not change its right half.
		To show that the union in (\ref{eq:ENDunion}) consists of sets which are almost pairwise disjoint, we want to count the number of such sets in which a given vertex in $V_1(P)\cup V_2(P)$ is contained.
		Then we can show that with high probability this count is at most $2$ for all relevant vertices.
		Given an integer $t'$ and a vertex $x \in V_1(P)\cup V_2(P)$ we let $Z^{x,t'}$ be the random variable that counts the number of pairs $(u,y)$, with $u \in \END_{\ell}^{t'}$ and $y \in \IN{u}{V_2(P^u)}$, such that $x \in \IN{y_+}{V_1(P^u)}$.
		For an integer $k \ge 1$, let $Z^{t'}_k$ be the random variable that counts the number of vertices $x \in V_1(P)\cup V_2(P)$ for which $Z^{x,t'} \ge k$.
		
		We show that the following two statements hold.
		\begin{enumerate}
			\item[(1)] If (\ref{eq:ENDqell}) holds for $t-1$, then with probability $1 - o\left(\frac{1}{n\log^2 n} \right)$ we have that $Z^{t-1}_3 = 0$.
			\item[(2)] If (\ref{eq:ENDqell}) holds for $t-1$ and $Z^{t-1}_3 = 0$, then with probability $1 - o\left(\frac{1}{n\log n} \right)$ we have that (\ref{eq:ENDqell}) holds for $t$.
		\end{enumerate}
		Note that since $q \ge \frac{\log^3 n}{n}$ we get that (\ref{eq:qp}) implies $t \le \frac{\log n}{4\log\log n}$.
		Assuming (1), and by taking a union bound over all $t = o(\log n)$ steps we get that
		\begin{align*}
		\Pr\left[Z^{t'}_3 = 0 ,\, \forall t' \in \{0, \ldots, t-1 \} \right] = 1 - t\cdot o\left(\tfrac{1}{n\log^2 n} \right) = 1 - o\left(\tfrac{1}{n\log n} \right).
		\end{align*}
		The induction step then follows immediately from (2).
		Hence it is left to prove (1) and (2).
		
		\paragraph*{Proof of (1).}
		Let $u \in \END_{\ell}^{t-1}$ and let $P^u \in \mathcal P_{\ell}^{t-1}$ be a path for which $u$ is its left endpoint (note that if $t=1$ then $u=u_1$ and $P^u=P$).
		Recall that $P$ contains no $\left(P, \frac{1}{8} \right)$-heavy vertices, implying that
		\begin{align*}
		\tfrac{1}{8}m \le \left|E'\left(V_2(P^u), u \right) \right| \le \tfrac{1}{4}m && \text{ and } && \tfrac{1}{8}m \le \left|E'\left(V_1(P^u), y_+ \right) \right| \le \tfrac{1}{4}m.
		\end{align*}
		Moreover, when performing online sprinkling left rotations to $P^u$ with $E'$, we have that
		\begin{align*}
		&\left|\IN{u}{V_2(P^u)} \right| \sim \Bin\left(\left|E'(V_2(P^u), u) \right|, \tfrac{100\log n}{m} \right) \\
		&\left|\IN{y_+}{V_1(P^u)} \right| \sim \Bin\left(\left|E'(V_1(P^u), y_+) \right|, \tfrac{100\log n}{m} \right),    
		\end{align*}
		for every $y \in \IN{u}{V_2(P^u)}$, and thus
		\begin{align*}
		12 \log n &\le \mathbb E \left[\left|\IN{u}{V_2(P^u)} \right| \right] \le 25 \log n, \\
		12 \log n &\le \mathbb E \left[\left|\IN{y_+}{V_1(P^u)} \right| \right] \le 25 \log n.
		\end{align*}
		By \Cref{Chernoff} with a suitable choice of $\eta$, given $u \in \END_{\ell}^{t-1}$ and $y \in \IN{u}{V_2(P^u)}$, we have
		\begin{align}
		\begin{split}
		\label{eq:Pexpsize}
		&\Pr\left[3\log n \le \left|\IN{u}{V_2(P^u)} \right| \le 50\log n \right] \ge 1 - 2n^{-\frac{27}{8}}, \\
		&\Pr\left[3\log n \le \left|\IN{y_+}{V_1(P^u)} \right| \le 50\log n \right] \ge 1 - 2n^{-\frac{27}{8}}.
		\end{split}
		\end{align}
		By taking a union bound over all $u \in \END_{\ell}^{t-1}$ we get
		\begin{align*}
		\Pr\left[\forall u \in \END_{\ell}^{t-1},\; 3\log n \le \left|\IN{u}{V_2(P^u)} \right| \le 50 \log n \right] &\ge 1 - 2\left|\END_{\ell}^{t-1} \right| \cdot n^{-\frac{27}{8}}.
		\end{align*}
		Recall that (\ref{eq:qp}) implies $t \le \frac{\log n}{4\log\log n}$.
		Since (\ref{eq:ENDqell}) holds for $t-1$ we have that
		\begin{align}
		\begin{split}
		\label{eq:p13}
		2\left|\END_{\ell}^{t-1} \right|\cdot n^{-\frac{27}{8}} &\le 2\left(50\log n \right)^{2(t-1)}\cdot n^{-\frac{27}{8}} \\
		&\le \left(\log n \right)^{2t} n^{-\frac{27}{8} + o(1)} \\
		&\le n^{-\frac{25}{8} + o(1)},
		\end{split}
		\end{align}
		which with the above implies that
		\begin{align}
		\label{eq:Pusize}
		\Pr\left[\forall u \in \END_{\ell}^{t-1},\; 3\log n \le \left|\IN{u}{V_2(P^u)} \right| \le 50\log n \right] \ge 1 - n^{-\frac{25}{8}+o(1)}.
		\end{align}
		
		Now note that given $u \in \END_{\ell}^{t-1}$ and $y \in \IN{u}{V_2(P^u)}$, for every $x \in V_1(P^u)$ we have $\Pr\left[x \in \IN{y_+}{V_1(P^u)} \right] = \frac{100\log n}{m}$.
		Moreover, given $u_1, u_2 \in \END_{\ell}^{t-1}$, $y_1 \in \IN{u_1}{V_2(P^{u_1})}$ and $y_2 \in \IN{u_2}{V_2(P^{u_2})}$ such that $(u_1, y_1) \neq (u_2,y_2)$, the events $\left\{x \in \IN{(y_1)_+}{V_1(P^{u_1})} \right\}$ and $\left\{x \in \IN{(y_2)_+}{V_1(P^{u_2})} \right\}$ are independent.
		Hence, given $x \in V_1(P) \cup V_2(P)$, by taking a union bound over all $u_1, u_2, u_3 \in \END_{\ell}^{t-1}$ and $y_i \in \IN{u_i}{V_2(P^{u_i})}$ for $i=1,2,3$, by (\ref{eq:Pusize}) and since (\ref{eq:ENDqell}) holds for $t-1$, we have
		\begin{align*}
		\Pr \left[Z^{x,t-1} \ge 3 \right] &\le \sum_{\substack{u_i \in \END_{\ell}^{t-1} \\ i=1,2,3}} \Pr\left[\forall i\in[3]\; \exists y_i \in \IN{u_i}{V_2(P^{u_i})} \text{ s.t. } x \in \bigcap_{i=1}^3 \IN{(y_i)_+}{V_1(P^{u_i})} \right] \\
		&\le \sum_{\substack{u_i \in \END_{\ell}^{t-1} \\ i=1,2,3}} \sum_{\substack{y_i \in \IN{u_i}{V_2(P^{u_i})} \\i=1,2,3}} \Pr\left[x \in \bigcap_{i=1}^3 \IN{(y_i)_+}{V_1(P^{u_i})} \right] \\
		&\le \sum_{\substack{u_i \in \END_{\ell}^{t-1} \\ i=1,2,3}} \left|\IN{u_1}{V_2(P^{u_1})} \right| \left|\IN{u_2}{V_2(P^{u_2})} \right| \left|\IN{u_3}{V_2(P^{u_3})} \right| \cdot \left(\tfrac{100\log n}{m} \right)^3 \\
		&\le \left|\END_{\ell}^{t-1} \right|^3 \left(\left(50\log n \right)^3 +  \left(\tfrac{1}{4}m \right)^3 \cdot n^{-\frac{25}{8}+o(1)} \right) \left(\tfrac{100\log n}{m} \right)^3 \\
		&\le 2\left(50\log n \right)^{6(t-1)} \left(100\log n \right)^3 \left(\tfrac{100\log n}{m} \right)^3 \\
		&\le \tfrac{64}{m^3}\left(50\log n \right)^{6t}.
		\end{align*}
		
		Lastly, we take a union bound over all $x \in V_1(P)\cup V_2(P)$.
		Recall that $t \le \frac{\log n}{4\log\log n}$ and $m \ge \frac{n}{4\log n}$, so we get
		\begin{align*}
		\Pr\left[Z^{t-1}_3 \ge 1 \right] &\le \sum_{x \in V_1\cup V_2} \Pr\left[Z^{x,t-1} \ge 3 \right] \\
		&\le \tfrac{32}{m^2} \left(50\log n \right)^{6t} \\
		&\le \tfrac{512\log^2 n}{n^2} \cdot n^{\frac{3}{4}+o(1)} \\
		&= n^{-\frac{5}{4}+o(1)} \\
		&= o\left(\tfrac{1}{n\log^2 n} \right),
		\end{align*}
		completing the proof of (1).

		\paragraph*{Proof of (2).}
		Having $Z^{t-1}_3 = 0$ is equivalent to having $Z^{x,t-1} \le 2$ for all $x \in V_1(P) \cup V_2(P)$.
		Thus, by (\ref{eq:ENDunion}) we have
		\begin{align}
		\label{eq:unionlwr}
		\left|\END_{\ell}^t \right| \ge \tfrac{1}{2} \sum_{u \in \END_{\ell}^{t-1}} \sum_{y \in \IN{u}{V_2(P^u)}} \left|\IN{y_+}{V_1(P^u)} \right|
		\end{align}
		and
		\begin{align}
		\label{eq:unionuppr}
		\left|\END_{\ell}^t \right| \le \sum_{u \in \END_{\ell}^{t-1}} \sum_{y \in \IN{u}{V_2(P^u)}} \left|\IN{y_+}{V_1(P^u)} \right|.
		\end{align}
		Hence, we wish to approximate the typical size of $\left|\IN{y_+}{V_1(P^u)} \right|$.
		
		Similarly to the argument yielding (\ref{eq:Pusize}), by (\ref{eq:Pexpsize}) we have that the probability
		\[\Pr\left[\forall u \in \END_{\ell}^{t-1} \; \forall y\in \IN{u}{V_2(P^u)} \; 3\log n \le \left|\IN{y_+}{V_1(P^u)} \right| \le 50\log n \right] \]
		is at least
		\begin{align*}
		1 - 2\sum_{u\in \END_{\ell}^{t-1}} \sum_{y\in \IN{u}{V_2(P^u)}} n^{-\frac{27}{8}} &\ge 1 - 2n^{-\frac{27}{8}} \sum_{u \in \END_{\ell}^{t-1}}\left|\IN{u}{V_2(P^u)} \right|.
		\end{align*}
		By (\ref{eq:Pusize}) and since (\ref{eq:ENDqell}) holds for $t-1$ and $t \le \frac{\log n}{4\log\log n}$, we have
		\begin{align*}
		2n^{-\frac{8}{27}}\sum_{u \in \END_{\ell}^{t-1}} \left|\IN{u}{V_2(P^u)} \right| &\le 2n^{-\frac{27}{8}}\left|\END_{\ell}^{t-1} \right| \left(50\log n + \tfrac{1}{4}m \cdot n^{-\frac{25}{8}+o(1)} \right) \\
		&\le n^{-\frac{27}{8}}\left(50\log n \right)^{2(t-1)} 100\log n \\
		&\le n^{-\frac{27}{8}} \left(50\log n \right)^{2t} \\
		&\le n^{-\frac{25}{8}+o(1)}.
		\end{align*}
		With the above this implies
		\begin{align}
		\begin{split}
		\label{eq:Puvsize}
		\Pr\left[\forall u \in \END_{\ell}^{t-1} \; \forall y\in \IN{u}{V_2(P^u)} \; 3\log n \le \left|\IN{y_+}{V_1(P^u)} \right| \le 50\log n \right] \ge 1 - n^{-\frac{25}{8}+o(1)}.
		\end{split}    
		\end{align}
		Hence, by (\ref{eq:Pusize}), (\ref{eq:unionuppr}) and (\ref{eq:Puvsize}) we get that
		\begin{align*}
		\left|\END_{\ell}^t \right| &\le \sum_{u \in \END_{\ell}^{t-1}} \sum_{y \in \IN{u}{V_2(P^u)}} \left|\IN{y_+}{V_1(P^u)} \right| \\
		&\le \sum_{u \in \END_{\ell}^{t-1}} \left|\IN{u}{V_2(P^u)} \right| \cdot 50\log n \\
		&\le \left|\END_{\ell}^{t-1} \right| \left(50\log n \right)^2
		\end{align*}
		holds with probability $1 - n^{-\frac{25}{8}+o(1)} = 1 - o\left(\frac{1}{n\log n} \right)$.
		Similarly, by (\ref{eq:Pusize}), (\ref{eq:unionlwr}) and (\ref{eq:Puvsize}) we get that
		\begin{align*}
		\left|\END_{\ell}^t \right| &\ge \tfrac{1}{2}\sum_{u \in \END_{\ell}^{t-1}} \sum_{y \in \IN{u}{V_2(P^u)}} \left|\IN{y_+}{V_1(P^u)} \right| \\
		&\ge \sum_{u \in \END_{\ell}^{t-1}} \left|\IN{u}{V_2(P^u)} \right| \cdot \tfrac{3}{2}\log n \\
		&\ge \left|\END_{\ell}^{t-1} \right| \left(\tfrac{3}{\sqrt{2}} \log n \right)^2 \\
		&\ge \left|\END_{\ell}^{t-1} \right| \left(2\log n \right)^2
		\end{align*}
		also holds with probability $1 - o\left(\frac{1}{n\log n} \right)$.
		Together with the assumption that (\ref{eq:ENDqell}) holds for $t-1$, we get that with probability $1 - o\left(\frac{1}{n\log n} \right)$ we have
		\[ \left(2\log n \right)^{2t} \le \left|\END_{\ell}^t \right| \le \left(50\log n \right)^{2t}, \]
		finishing the proof of (2).
	\end{proof}
	
	Given \Cref{cl:ENDqell}, it follows immediately that with probability $1 - o\left(\frac{1}{n\log n} \right)$ there exists $t_{\ell}$ which satisfies the conditions of the claim and for which we have
	\begin{align*}
	\left|\END_{\ell}^{t_{\ell}} \right| = \Omega\left(\tfrac{\log n}{\sqrt{q}} \right).
	\end{align*}
	Morover, $t_{\ell}$ satisfies (\ref{eq:qp}), implying $t_{\ell} \le \frac{\log n}{4\log\log n}$ and completing the proof of the lemma.
\end{proof}

\section{\texorpdfstring{Proof of \Cref{lem:moment}}{Proof of Lemma 5.2}}
\label{sec:Appendix3Moment}

Recall the bijection between balanced bipartite graphs and directed graphs (digraphs) determined by a permutation $\pi$. 
For a bipartite graph $B$, with parts $X$ and $Y$, which are two disjoint copies of $[n]$, we write $xy$ for the edge between vertex $x\in X$ and $y\in Y$ (in particular, $xy \neq yx$). Given any permutation $\pi \in S_n$, the digraph $D' = D_{\pi}:=D_{\pi}(B)$ has vertex set $[n]$ and edge set $E\left(D_\pi(B) \right) \coloneqq \left\{\overrightarrow{x \pi(y)} ~:~ xy \in E(B) , x \neq \pi(y)\right\}$.
Let $\mathcal{M}=\{M_1 \ldots M_{\delta}\}$ be a collection of $\delta$ edge-disjoint perfect matchings of $B$, where $\delta \coloneqq \delta^{\pm}\left(D' \right) = (1+o(1))np$, and $\frac{\log^{15} n}{n} \le p \le \varepsilon$ (see \Cref{lem:mindegDnp}).
For each perfect matching $M_i\in \mathcal M$, the directed graph induced by the edges of $M_i$ under the bijection, denoted by $F_i \coloneqq D_\pi(M_i)$, is a $1$-factor in $D_\pi(B)$. 
For each $M_i$, define $m_i \in S_n$ as the permutation where $m_i(x)=y$ when $xy \in M_i$.
Thus, the $1$-factor $F_i$ is spanned by directed edges of the form $\overrightarrow{x \pi(m_i(x))}$ for $x \in X$.

For a vertex $v \in [n]$ and $1 \leq i \leq \delta$ we denote by $C_{v,i}$ the (unique) cycle in $D_\pi(M_i)$ that contains $v$, and we let $c_{v,i} \coloneqq |C_{v,i}|$ denote its length. In this section, we will consider isolated vertices to be cycles of length 1.
Note that since $\pi \in S_n$ is a uniformly random permutation, both $C_{v,i}$ and $c_{v,i}$ are random variables.

We restate the lemma proved in this appendix.
\begin{replemma}{lem:moment}
	For every $v \in [n]$ we have
	\begin{align*}
	\mathbb E \left(\left(\sum_{i \in [r]} \frac{1}{c_{v,i}} \right)^3 \right) = O\left(p \log^3 n \right).
	\end{align*}
\end{replemma}

Before we prove Lemma \ref{lem:moment}, we need an auxiliary result that allows us to compute the probability that two or three cycles simultaneously have a prescribed size.

\begin{lemma}\label{lem:highmoments}
    Let $2\leq k\leq 3$ and $1\leq a_1\leq \ldots a_k \leq n/6$. Then for every $v \in [n]$ and collection of $k$ matchings $M_{i_1},\ldots, M_{i_k}$ we have that
\begin{align*}
    \PP\left(\bigcap_{j=1}^k (c_{v,i_{j}}=a_j)\right)=O(n^{-k}).
\end{align*}
\end{lemma}

\begin{proof}
Fix $v \in [n]$. For simplicity, suppose that the collection of $k$ matchings is $M_1, \ldots, M_k$ and write $C_i$ and $c_i$ instead of $C_{v,i}$ and $c_{v,i}$. Throughout the proof we will only present the arguments for the case $k=3$. The corresponding claims for $k=2$ should follow similarly. Let $\cE(a_1,a_2,a_3)= (c_1=a_1)\cap (c_2=a_2) \cap (c_3=a_3)$ be the event that the cycle $c_i$ has length $a_i$ and let $\cE(a_1,a_2)=(c_1=a_1)\cap (c_2=a_2)$ the corresponding event for two cycles. Our goal is to prove that $\PP(\cE(a_1,a_2,a_3))=O(n^{-3})$ and $\PP(\cE(a_1,a_2))$.

 Define the random subset $T:=V\left(C_i\cup C_j\cup C_k\right)\setminus \{v\}$. By a slight abuse of notation, we also consider $T$ to be a subset of $X$. Since the matchings $M_1,M_2$, and $M_3$ are edge-disjoint, we observe that in the bipartite graph $G:=M_1\cup M_2\cup M_3$, the following properties hold:
 \begin{enumerate}
     \item $e_G\left(T,\pi^{-1}(T)\right)\geq c_1+c_2+c_3$, and
     \item the subgraph of $G$ induced by $(T\cup \{v\})\cup\pi^{-1}(T\cup\{v\})$ has no isolated vertices, and
     \item $\pi^{-1}(v)\in Y \textrm{ has three neighbors in } T \cup \{v\}.$
 \end{enumerate}

 These observations are sufficient to prove the following claim:

\begin{claim}\label{clm:cyclesbig}
    $\Prob(|T|=O(1))=O(n^{-3}).$
\end{claim} 

\begin{proof}
    Conditioning on the size of $T$, we distinguish between two cases: 
    \paragraph{\bf{Case 1.}} $v\pi^{-1}(v)\notin E(G)$. In this case, there are $n-3$ possible choices for $y\in Y$, where $y=\pi^{-1}(v)$. The probability that $\pi(y)=v$ is $\frac{1}{n}$. Let $N\coloneqq N_G(y)$ denote the set of the three neighbors of $y$ in $G$. Next, choose a subset $S\subseteq X \setminus \left(\{v\}\cup N\right)$ of size $|T|-3$, and define $T:= N\cup S$. 
    
    For each $x\in T$, select a non-empty subset of neighbors $S_x\subseteq N_G(x)$, and let $T':=\cup_{x\in T} S_x\subseteq Y$ be the obtained set. Clearly, there are at most $7^T$ ways to choose $T'$ of size $|T|$ such that, if $\pi(T'\cup \{y\})=T\cup\{v\}$, then its image under $D_{\pi}$ will give us three cycles containing $v$ on $T$. The probability that $\pi(T')=T$ is at most $\frac{T!}{(n)_{T}}$. For $|T|=O(1)$ this is of order $O\left(\left(\frac{1}{n}\right)^{T}\right)$.

Therefore, the probability that the edges $E(T\cup \{v\},\pi^{-1}(T)\cup\{y\})$ map to $T\cup \{v\}$ in $D_{\pi}$, forming three cycles containing $v$, is at most 
$$(n-3)\cdot \frac{1}{n}\cdot \binom{n}{|T|-3}\cdot 7^{|T|}\cdot O(n^{-|T|})=O(n^{-3})$$
as desired. 
\paragraph{\bf{Case 2.}}  $v\pi^{-1}(v)\in E(G)$. This case is quite similar to the previous one, with a few notable differences. Here, there are only $3$ possible choices for $y=\pi^{-1}(v)$, since  $y$ must be chosen as a neighbor of $v$. After selecting $y$ and including its two additional neighbors in $T$, we need to choose a subset $S\subseteq X$ of size $|T|-2$ (instead of $|T|-3$ as in the previous case) extra vertices to complete the set $T$.

Thus, the probability that the edges $E(T\cup \{v\},\pi^{-1}(T)\cup\{y\})$ are mapped to $T\cup \{v\}$ in $D_{\pi}$, forming three cycles containing $v$, is at most 
$$3\cdot \frac{1}{n}\cdot \binom{n}{|T|-2}\cdot 7^{|T|}\cdot O(n^{-|T|})=O(n^{-3})$$
as desired. 

To complete the proof of the claim take an upper bound over all possible values of $|T|$. Since $|T|=O(1)$, the result follows.
\end{proof}

By defining $T_2:=V(C_1\cup C_2)\setminus\{v\}$, a similar argument yields the following bound on the probability that the two shortest cycles have constant length.

\begin{claim}\label{clm:big2}
    $\Prob\left(|T_2|=O(1)\right)=O(n^{-2}).$
\end{claim}

The following claim demonstrates that it is highly unlikely for the three cycles containing $v$ to have significant overlap.

\begin{claim} \label{large intersection}
Let $1\leq a_1\leq a_2\leq a_3$ be such that $\omega(1)=a_1+a_2+a_3\leq n,$. Then, 
    $$\Prob\left((|T|\leq a_1+a_2+a_3/2)\cap \mathcal E(a_1,a_2,a_3) \right)=n^{-\omega(1)}.$$
\end{claim}

\begin{proof}
Suppose that $|T|\leq a_1+a_2+a_3/2$ and $\mathcal E(a_1,a_2,a_3)$ holds. In this case, we must have $e_G(T,\pi^{-1}(T))\geq a_1+a_2+a_3.$ Let $x$ be the number of vertices of degree greater than $1$ in the subgraph of $G$ induced by $T\cup \pi^{-1}(T)$. Now, using the fact that the maximum degree in $G$ is $3$, we obtain that: 
$$a_1+a_2+a_3\leq 3x+a_1+a_2+a_3/2-x,$$
which simplifies to 
$$x\geq a_3/4.$$
Next, note that in this scenario, it is possible to greedily find a subset of $a_3/20=\omega(1)$ vertex-disjoint cherries (pairs of vertices sharing a common neighbor), with the centers of the cherries forming a subset $S\subseteq \pi^{-1}(T)$. This implies that $|N_G(S)|\geq 2|S|$.

Now, consider the number of pairs $(T,T')$ of subsets $T\subseteq X, T'\subseteq Y$ with $|T|=|T'|\leq a_1+a_2+a_3/2$ and there are at least $a_3/20$ vertices in $T'$ whose neighborhood in $T$ is of size at least $a_3/10$. This number is at most 
$$\binom{n}{a_3/20}\binom{n}{|T|-a_3/10}\leq \left(\frac{Cn}{T}\right)^{|T|-a_3/20},$$
for some constant $C>0$. 

Given $T$, there are at most $7^{|T|}$ ways to choose $T'$ (recall that the minimum degree in the graph induced by $T\cup \pi^{-1}(T)$ is at least $1$). Now, for each such pair $(T,T')$, the probability that $\pi(T')=T$ is at most $\binom{n}{|T|}^{-1}\leq (|T|/n)^{|T|}$. Therefore, we obtain that 

$$ \Prob\left((|T|\leq a_1+a_2+a_3/2)\cap \mathcal E(a_1,a_2,a_3) \right)\leq \left(\frac{Cn}{T}\right)^{|T|-a_3/20}\cdot 7^{|T|}\cdot \left(\frac{|T|}{n}\right)^{|T|}=n^{-\omega(1)}. $$

This completes the proof.
\end{proof}

If we consider only the case with two cycles $C_1,$ and $C_2$, a similar proof will give us the following.

\begin{claim}\label{cl:large2} Let $1\leq a_1\leq a_2$ be such that $\omega(1)=a_1+a_2\leq n,$. Then, 
    $$\Prob\left((|T_2|\leq a_1+a_2/2)\cap \mathcal E(a_1,a_2)\right)=n^{-\omega(1)}.$$
\end{claim}

It will be helpful to describe the following exposure procedure for revealing the values of the (random) permutation $\pi: [n] \to [n]$. This procedure is designed to track and extend the directed path in the directed graph $D_{\pi}$ that contains $v$, and ultimately reveal the entire structure of the cycles $C_1$, and then $C_2$ and $C_3$. Given a set of labeled vertices $U\subseteq Y$ and a matching $M_i$, let $M_i[X\cup U]$ be the induced subgraph of $M_i$ on the vertex set $X\cup U$ (where we view $U$ as a subset of $Y$). We define $D[M_i,U]:=D_\pi(M[X\cup U])$. We proceed by exposing the labels of the vertices of $Y$ one step at the time as follows:

\begin{enumerate}
    \item Initialize by setting $U=\emptyset$ (the set of labeled vertices).
    \item For $1\leq i \leq 3$, do the following:
    \begin{enumerate}
        \item[(i)] Let $P$ be the directed component containing the vertex $v$ in $D[M_i,U]$.
        \item[(ii)] While $P$ is not a directed cycle, do the following step: Note that the component $P$ is either a directed path or an isolated vertex. Let $x$ be the sink of $P$ and let $y \in Y$ be the vertex adjacent to $x$ in $M_i$. Since $x$ is a sink, the value of $\pi(y)$ was never revealed. We reveal $\pi(y)$, update $U:=U\cup\{y\}$ and $P$, and repeat.
        \item[(iii)] If $P$ is a directed cycle, we move to the next matching ($i:=i+1$).
    \end{enumerate}
    \item Expose the remaining labels in $[n]\setminus U$ in arbitrary order.
\end{enumerate}

This procedure incrementally reveals the permutation $\pi$ by following the edges of the matchings $M_1,M_2$, and $M_3$
  that define the cycles $C_1,C_2,$ and $C_3$, repectively, in $D_{\pi}$. The key idea is to start at the vertex $v$, trace its unique path in $D_{\pi}(M_{s})$ for $s\in \{i,j,k\}$, and extend the corresponding directed path by uncovering the connections dictated by matchings. For each cycle, we reveal the labels of vertices one step at a time until the entire cycle structure is determined. Once all relevant cycles are exposed, the remaining vertices are labeled arbitrarily. 

Let us now introduce some useful notation. Denote by $t_1,t_2,$ and $t_3$, the closing times of the cycles $C_1,C_2,$ and $C_3,$ respectively, according to the exposure procedure described above. Notice that we always have $t_1=c_1\leq t_2,t_3$, although  it is possible that $t_3<t_2$. At the closing time of a cycle $C_s$, where $s\in [3]$, if we examine the (unique) directed path containing $v$ in the image of the matching $M_s$ in $D_{\pi}$, then the following hold: 
\begin{enumerate}
    \item the path is of length $c_s-1$, and
    \item there is exactly one vertex $w$ (the starting point of this path) that remains unlabeled at this stage; that is, $\pi^{-1}(w)$ is yet undefined. 
\end{enumerate}
 Note that if $z$ is the other endpoint of the path, and we consider $z\in X$, then in order to close the cycle, its unique neighbor in $M_s$, if yet unlabeled, must be labeled $w$. At this point, if there are still $m$ unlabeled vertices in $Y$, the conditional probability that $\pi(z)=w$ is $\frac{1}{m}$.

This observation can be slightly generalized to provide an upper bound on the probability that the three cycles $C_1,C_2$ and $C_3$ have distinct closing times.

\begin{claim}\label{clm:difclosing}
   Let $a_1\leq a_2\leq a_3$ be such that $\omega(1)=a_1+a_2+a_3\leq n/2$. Let $\mathcal A=(t_1 <t_2<t_3)$ be the event that the three cycles close sequentially at distinct times. Then,
   $$\Prob\left(\mathcal E(a_1,a_2,a_3)\cap \mathcal A\right)=O(n^{-3}).$$
\end{claim}

\begin{proof}
Let $\cE^{(3)}:=\cE(a_1,a_2,a_3)\cap \cA$, $\cE^{(2)}:=\cE(a_1,a_2)\cap (t_1<t_2)$ and $\cE^{(1)}:=(c_1=a_1)$. Fix an instance of $\cE^{(3)}$. For $1\leq i\leq 3$, let $U_i$ be the ordered set of labels revealed from steps $t_{i-1}+1$ to $t_i-1$ (where $t_0=0$). Note by the description of the procedure and the observation preceding the claim, that such sets $U_1\cup U_2\cup U_3$ uniquely determines the instance of $\cE^{(3)}$.

Let $P_3$ be the path containing $v$ using the edges of $M_3$ during the procedure. Note that at each step of the process after $t_{2}$, the length of $P_3$ increases in at least one edge. Thus, conditioned on fixed $U_1$ and $U_2$, all the events leading to the possible sequence of vertices $U_{3}$ until step $t_{3}-1$ are disjoint. Moreover, as discussed earlier, the probability of closing the cycles at step $t_3$ conditioned on $U_3$ is at most $2/n$ (since, by assumption, there are always at least $n/2$ unlabeled vertices at any given stage of the procedure). Therefore, 
\begin{align*}
    \PP(\cE^{(3)}\mid U_1,U_2)=\sum_{U_3}\PP(\cE^{(3)}\mid U_1,U_2,U_3)\PP(U_3)\leq\frac{2}{n}\sum_{U_3}\PP(U_3)=\frac{2}{n}.
\end{align*}
Consequently, we obtain that
\begin{align*}
    \PP(\cE^{(3)})=\sum_{U_1,U_2}\PP(\cE(a_1,a_2,a_3)\mid U_1,U_2)\PP(U_1, U_2)=\frac{2}{n}\sum_{U_1,U_2}\PP(U_1, U_2)\leq\frac{2}{n}\PP(\cE^{(2)}).
\end{align*}
A similar argument shows that
\begin{align*}
    \PP(\cE^{(2)})\leq \frac{2}{n}\PP(\cE^{(1)}).
\end{align*}
Since $\PP(\cE^{(1)})=\frac{1}{n}$, by putting together the two inequalities, we have that $\PP(\cE^{(3)})=O(n^{-3})$. This concludes the proof.
\end{proof}

The last piece of the puzzle that we need is the following claim, which basically combines all the above: 

\begin{claim}
    Let $1 \leq a_1\leq a_2 \leq a_3\leq n/6$. Then $$\Prob\left(\mathcal E(a_1,a_2,a_3)\right)=O(n^{-3}).$$
\end{claim}

\begin{proof}
    We compute the probability by splitting into several cases depending on the times that each cycle close. Indeed, note that since $t_1\leq t_2,t_3$, the event $\cE(a_1,a_2,a_3)$ can be written as the disjoint union of the events
    \begin{align*}
        \cE(a_1,a_2,a_3)=\cE_1\cup \cE_2 \cup \cE_3 
    \end{align*}
    where we define the events $\cE_i$ by
    \begin{align*}
        \cE_1&:=\cE(a_1,a_2,a_3)\cap (t_3\leq t_2)\\
        \cE_2&:=\cE(a_1,a_2,a_3)\cap  (t_1=t_2<t_3)\\
        \cE_3&:=\cE(a_1,a_2,a_3)\cap (t_1 < t_2 < t_3).
    \end{align*}
    We bound $\PP(\cE_i)$ for $1\leq i \leq 3$.
    \paragraph{\textbf{Case 1: $\cE_1$}} Let $S=V\left(C_1\cup C_2\cup C_3\right)$. Note that in this case $|S|\leq a_1+a_2<a_1+a_2+a_3/2$. Hence, by Claim \ref{clm:cyclesbig} and \ref{large intersection}, we obtain that
    \begin{align*}
        \PP(\cE_1)=\PP(\cE_1\cap (|S|=O(1)))+\PP(\cE_1\cap (|S|=\omega(1)))=O(n^{-3})+O(n^{-\omega(1)})=O(n^{-3}).
    \end{align*}

    \paragraph{\textbf{Case 2: $\cE_2$}} Let $S=V\left(C_1\cup C_2\right)$. Note that in this case $|S|=a_1\leq a_1+a_2/2$. Let $\tilde{\cE}_2$ be the event $\cE(a_1,a_2)\cap (t_1=t_2)$. Thus, by Claim \ref{clm:big2} and \ref{cl:large2}, we obtain that
    \begin{align*}
        \PP(\tilde{\cE_2})=\PP(\tilde{\cE_2}\cap (|S|=O(1)))+\PP(\tilde{\cE_2}\cap (|S|=\omega(1)))=O(n^{-2}).
    \end{align*}
    Let $U$ be the ordered set of labels revealed from steps $t_2$ to step $t_3-1$ conditioned on a instance $\omega \in \tilde{\cE_2}$. By a similar argument as in Claim \ref{clm:difclosing}, we obtain that
    \begin{align*}
        \PP(\cE_2\mid \omega)=\sum_{U}\PP(\cE_2\mid \omega, U)\PP(U)\leq \frac{2}{n}. 
    \end{align*}
    Therefore,
    \begin{align*}
        \PP(\cE_2)=\sum_{\omega \in \tilde{\cE_2}}\PP(\cE_2\mid \omega)\PP(\omega)\leq \frac{2}{n}\PP(\tilde{\cE_2})=O(n^{-3}).
    \end{align*}

    \paragraph{\textbf{Case 3: $\cE_3$}}
    We obtain that $\PP(\cE_3)=O(n^{-3})$ immediately from Claim \ref{clm:difclosing}.

    By putting all the cases together, we obtain that 
    \begin{align*}
    \PP(\cE(a_1,a_2,a_3))=O(n^{-3})
    \end{align*}
    as desired.
\end{proof}

We note that a similar bound can be obtained for two cycles. More precisely, one can prove that $\PP(\cE(a_1,a_2))=O(n^{-2})$. This finishes the proof of the lemma.
\end{proof}

We are now able to prove Lemma \ref{lem:moment}.

\begin{proof}[Proof of Lemma \ref{lem:moment}]
Let $v \in [n]$. For simplicity, we will write $C_i$ and $c_i$ instead of $C_{v,i}$ and $c_{v,i}$, respectively. We begin with an auxiliary result.

\begin{claim}\label{clm:expect}
The following three bound holds:
\begin{enumerate}
		\item For every $i \in [\delta]$, 
		$\mathbb E\left(\frac{1}{c^3_{i}} \right) = O\left(\frac{1}{n} \right)$.
		\item For every distinct $i,j \in [\delta]$,
		$\mathbb E \left(\frac{1}{c^2_{i} c_{j}} \right) = O\left(\frac{\log^2 n}{n^2} \right).$
		
		\item For every distinct $i,j,k \in [\delta]$,
		$\mathbb E \left(\frac{1}{c_{i} c_{j} c_{k}} \right) = O \left(\frac{\log^3 n}{n^3} \right)$.
\end{enumerate}
\end{claim}

\begin{proof}
We start by showing statement (1). Note that for a single cycle, the probability $\PP(c_i=t)=1/n$. Hence,
\begin{align*}
    \EE\left(\frac{1}{c_i^3}\right)=\sum_{a_i\in [n]}\frac{1}{a_i^3}\PP(c_i=a_i)=\frac{1}{n}\sum_{a_i\in [n]}\frac{1}{a_i^3}=O\left(\frac{1}{n}\right).
\end{align*}

To prove statement (2) and (3) we use Lemma \ref{lem:highmoments}. Indeed, for statement (2) we have that
\begin{align*}
    \EE\left(\frac{1}{c_i^2c_j}\right)&=\sum_{a_i,a_j \in [n]}\frac{1}{a_i^2a_j}\PP\left((c_i=a_i)\cap(c_j=a_j)\right)\\ &\leq \sum_{a_i,a_j \in [n]}\frac{1}{a_ia_j}\PP\left((c_i=a_i)\cap(c_j=a_j)\right)\leq \sum_{a_i\leq a_j}\frac{2}{a_ia_j}\PP\left((c_i=a_i)\cap(c_j=a_j)\right)\\
    &\leq\sum_{a_i\leq a_j\leq n/4}\frac{2}{a_ia_j}\PP\left((c_i=a_i)\cap(c_j=a_j)\right)+\sum_{a_i\in [n]}\frac{8}{na_i}\PP\left(c_i=a_i\right)\\
    &=O(n^{-2})\left(\sum_{a_i\leq a_j\leq n/4}\frac{1}{a_ia_j}+\sum_{a_i\in [n]}\frac{1}{a_i}\right)=O\left(\frac{\log^2 n}{n^2}\right),
\end{align*}
where we use Lemma \ref{lem:highmoments} for two cycles.

A similar computation holds for statement (3):
\begin{align*}
    \EE\left(\frac{1}{c_ic_jc_k}\right)&=\sum_{a_i,a_j,a_k \in [n]}\frac{1}{a_ia_ja_k}\PP\left((c_i=a_i)\cap(c_j=a_j)\cap(c_k=a_k)\right)\\
    &\leq \sum_{a_i\leq a_j\leq a_k}\frac{6}{a_ia_ja_k}\PP\left((c_i=a_i)\cap(c_j=a_j)\cap(c_k=a_k)\right)
    =\Sigma_1+\Sigma_2+\Sigma_3,
\end{align*}
where
\begin{align*}
    &\Sigma_1:=\sum_{a_i\leq a_j\leq a_k\leq n/6}\frac{6}{a_ia_ja_k}\PP\left((c_i=a_i)\cap(c_j=a_j)\cap(c_k=a_k)\right)\\
    &\Sigma_2:= \sum_{\substack{a_i\leq a_j\leq n/6\\ a_k>n/6}}\frac{6}{a_ia_ja_k}\PP\left((c_i=a_i)\cap(c_j=a_j)\cap(c_k=a_k)\right), \textrm{ and }\\
    &\Sigma_3:=\sum_{\substack{a_i\in [n]\\ n/6<a_j\leq a_k}}\frac{6}{a_ia_ja_k}\PP\left((c_i=a_i)\cap(c_j=a_j)\cap(c_k=a_k)\right)
\end{align*}

We bound each of the $\Sigma_i$s separately as follows: 

By Lemma \ref{lem:highmoments} applied for three cycles, we know that for $a_1\leq a_2\leq a_3\leq n/6$, the probability $\PP((c_i=a_i)\cap (c_j=a_j)\cap (c_k=a_k))= O(n^{-3})$. Therefore,
\begin{align*}
    \Sigma_1=O\left(n^{-3}\sum_{a_1\leq a_2\leq a_3\leq n/6}\frac{1}{a_1a_2a_3}\right)=O\left(\frac{\log^3 n}{n^3}\right)
\end{align*}
Similarly, Lemma \ref{lem:highmoments} applied for two cycles gives that if $a_1\leq a_2\leq n/6$, then $\PP((c_i=a_i)\cap(c_j=a_j))=O(n^{-2})$. Hence
\begin{align*}
    \Sigma_2\leq \sum_{a_i\leq a_j\leq n/6}\frac{36}{na_ia_j}\PP\left((c_i=a_i)\cap(c_j=a_j)\right)= O\left(n^{-3}\sum_{a_i\leq a_j\leq n/6}\frac{1}{a_ia_j}\right)=O\left(\frac{\log^2 n}{n^3}\right)
\end{align*}
Finally, to bound $\Sigma_3$, we just need to use that $\PP(c_i=a_i)=1/n$. Therefore,
\begin{align*}
\Sigma_3 \leq \sum_{a_i\in [n]}\frac{216}{n^2a_i}\PP\left(c_i=a_i\right)=\frac{216}{n^3}\sum_{a_i\in [n]}\frac{1}{a_i}=O\left(\frac{\log n}{n^3}\right)
\end{align*}
By summing everything together, we obtain 
\begin{align*}
    \EE\left(\frac{1}{c_ic_jc_k}\right)=O\left(\frac{\log^3 n}{n^3}\right).
\end{align*}
This concludes the proof of the claim.
\end{proof}

We finish the proof of the proposition by applying Claim \ref{clm:expect} after expanding the third moment
\begin{align*}
	\mathbb E \left(\left(\sum_{i \in [\delta]}\frac{1}{c_{i}} \right)^3 \right) &= \sum_{i \in [\delta]} \mathbb E\left(\frac{1}{c_{i}^3} \right) + 3\sum_{\substack{i, j \in [\delta] \\ i \neq j}} \mathbb E \left(\frac{1}{c_{i}^2 c_{j}} \right) + 6\sum_{\substack{i,j,k \in [\delta] \\ i<j<k}} \mathbb E \left(\frac{1}{c_{i} c_{j} c_{k}} \right)\\
  &=\delta \cdot O\left(\frac{1}{n} \right) + \delta^2 \cdot O\left(\frac{\log^2 n}{n^2} \right) + \delta^3 \cdot O\left(\frac{\log^3 n}{n^3} \right) = O\left(p \log^3 n \right),
 \end{align*}
 as desired.
\end{proof}
	
\end{document}